\documentclass[12pt]{article} 
\usepackage{hyperref}
\usepackage{tikzit}
\usepackage{tikz, amsbsy}
\usetikzlibrary{decorations.pathreplacing,calligraphy}
\usepackage{hyperref}
\usepackage{caption,subcaption}
\usepackage{graphicx}
\usepackage{graphics}
\usepackage{comment}

\usepackage[]{graphicx}

\usepackage{tikz-cd}
\usepackage{tkz-euclide}
\usepackage{amssymb,amsmath,amsthm}
\usepackage{listings}
\newtheorem{thm}{Theorem}

\newtheorem{cor}{Corollary}
\newtheorem{conjecture}{Conjecture}
\newtheorem{prop}{Proposition}
\newtheorem{lemma}{Lemma}

\theoremstyle{definition}
\newtheorem{example}{Example}
\newtheorem{definition}{Definition}
\newtheorem{remark}{Remark}
\newtheorem*{algorithm}{Triangulation Algorithm}

\newcommand{\Z}{\mathbb{Z}}

\newcommand{\R}{\mathbb{R}}

\newcommand{\D}{\mathcal{D}}
\newdimen\R
\usepackage[style=alphabetic]{biblatex}
\usepackage{hyperref}

\title{Friezes over $\Z[\sqrt{2}]$} 
\author{Esther Banaian, Libby Farrell, Amy Tao, Kayla Wright, Joy Zhichun Zhang}

\begin{document}

\maketitle
\begin{abstract}
    A frieze on a polygon is a map from the diagonals of the polygon to an integral domain which respects the Ptolemy relation. Conway and Coxeter previously studied positive friezes over $\mathbb{Z}$ and showed that they are in bijection with triangulations of a polygon.  We extend their work by studying friezes over $\Z[\sqrt{2}]$ and their relationships to dissections of polygons. We largely focus on the characterization of unitary friezes that arise from dissecting a polygon into triangles and quadrilaterals. We identify a family of dissections that give rise to unitary friezes and conjecture that this gives a complete classification of dissections which admit a unitary frieze. 
\end{abstract}
    

\section{Introduction}

In this paper, we will study friezes. A frieze is a ring homomorphism from a cluster algebra $\mathcal{A}(Q)$ to an integral domain $R$. When the cluster algebra arises from surface $S$ with marked points $M$, the generators of the algebra correspond to arcs on the surface with relations provided by skein relations \cite{fomin2008cluster}. Therefore, a frieze from such a cluster algebra can instead be viewed as a map from the arcs on $(S,M)$ to $R$ which respects skein relations.  We will mainly take this latter point of view. 

The study of friezes in fact predates the study of cluster algebras. Finite frieze patterns were first studied by Coxeter in \cite{Coxeter}. Frieze patterns are certain arrays of numbers which satisfy a local relation (the \emph{diamond condition}). Finite frieze patterns and friezes on polygons are in bijection; we can interpret a frieze pattern as listing all the values of a frieze. We will usually use the language of a frieze in this article for notational convenience.

Conway and Coxeter showed that finite frieze patterns with entries in $\mathbb{Z}_{\geq 0}$ are in bijection with triangulated polygons \cite{CCFrieze}. Caldero and Chapoton show that these finite frieze patterns also have a connection to both cluster algebras of type $A$ and the module category of a path algebra from a type $A$ quiver \cite{caldero2006cluster}.
Friezes with values in $\mathbb{Z}_{\geq 0}$ associated to cluster algebras of types $\widetilde{A}$ and $D$ were studied in \cite{GunawanSchiffler} and \cite{FontainePlamondon} respectively. 

Recently, Holm and J{\o}rgensen investigated friezes associated to dissections of polygons \cite{HJ}. The sizes of the subpolygons involved in the dissection determine the integral domain the frieze takes values in. Holm and J{\o}rgensen show that there is a bijection between dissections that divide an $n$-gon into $p$-gons and friezes on an $n$-gon with values in $\mathbb{Z}[2\cos(\pi/p)]$ which send every diagonal of the form $(i,i+2)$to a multiple of $2\cos(\pi/p)$; that is, every such frieze can be seen as arising from a dissection. This leads to the natural question:  Is there a similar characterization of the friezes from more general dissections? In this article, we focus our attention on dissections into triangles and quadrilaterals. Friezes from such dissections were also studied by Andritsch, who compared them to friezes from certain triangulations \cite{andritsch2020note}. In Section \ref{sec:TypesOfFriezes}, we provide examples to show that some natural first choices for characterizations do not adequately describe the set of friezes from dissections into triangles and quadrilaterals. That is, for each characterization we exhibit an example of such a frieze which does not arise from a dissection.

 One type of frieze investigated in Section \ref{sec:TypesOfFriezes} is a \emph{unitary frieze}. We say that a frieze $f$ on a surface $S$ over a ring $R$ is unitary if there exists a triangulation $T = \{\tau_1,\ldots,\tau_n\}$ of $S$ such that each $f(\tau_i)$ is a unit in $R$. We let $R^\times$ denote the set of units in $R$. In Conjecture \ref{conj:UnitaryConj}, we propose that a dissection of a polygon into triangles and quadrilaterals will produce a unitary frieze if and only if we can decompose the dissection into \emph{towers}. A tower, defined in Section \ref{sec:Towers}, is a dissection of a polygon into a straight row of quadrilaterals with a triangle on one end. In Theorem \ref{thm:Towers}, we verify that every dissection which can be decomposed into towers provides a unitary frieze. 

We prove that the opposite direction of this conjecture is true in several cases. In Section \ref{sec:Separated}, we consider dissections where the sets of triangles and quadrilaterals are separated from each other. In Section \ref{sec:Type3}, we consider dissections where every vertex is adjacent to at most three subpolygons. Part of our proof of the result for the latter type of dissection (Theorem \ref{thm:TypeThreeNonTowerFails}) involves casework to show various types of arcs that form triangles with arcs from towers cannot have unit weight under the frieze from the dissection; more details about these cases are provided in Appendix \ref{sec:AppA}. 


\section{Friezes}\label{sec:TypesOfFriezes}

\subsection{Background}

We think of a \emph{polygon} or $n$\emph{-gon} as a finite set $V=\{0,1,\dots, n-1\}$ of vertices with the natural cyclic ordering.  Arcs in the polygon are denoted by $(i,j)$ for two vertices $i \neq j$. We say that two arcs $(i,j)$ and $(k,\ell)$ \emph{cross}  when either $i < k < j < \ell$ or $i < \ell < j < k$, working cyclically modulo $n$. Arcs of the form $(i,i+1)$ are boundary arcs and thus never cross any other arcs.  Note that, while some figures in this article will not look convex in order to stress certain patterns, we always will assume we are working with convex polygons.

While friezes can be defined in more general settings, our article focuses on friezes from polygons so we will define a frieze in this context. Our definition largely follows the definition of a frieze in \cite{HJ}.
    
    \begin{definition}\label{def:frieze}
    Let $P$ be a polygon with vertex set $V$ and let $R$ be an integral domain. A \textit{frieze on $P$} is a map $f:V\times V\to R$ where the following conditions are satisfied:
    \begin{enumerate}
        \item $f(i,j)=0$ if and only if $i=j$
        \item $f(i,i+1) = f(i,i-1) =1$
        \item $f(i,j)=f(j,i)$
        \item If $(i,j)$ and $(k,\ell)$  are crossing diagonals of $P$, then we have the Ptolemy relation $f(i,j) f(k,\ell)=f(i,\ell)f(j,k) + f(i,k)f(j,\ell)$. \end{enumerate}
        \begin{center}
\small{
	\begin{tikzpicture}[scale = 1]
	\draw (0,1) rectangle (3,3);
	\draw (0,1) -- (3,3);
	\draw (0,3) -- (3,1);
	\draw (-0.25, 1) node {$k$};
	\draw (-0.25, 3) node {$i$};
	\draw (3.25, 1) node {$j$};
	\draw (3.25, 3) node {$\ell$};
	\end{tikzpicture}}
\end{center}
    
    \end{definition}

 We refer to the value $f(i,j)$ as the \emph{weight} of the arc $(i,j)$.  We give an example of a frieze on a pentagon over $\mathbb{Z}$.     
    
\begin{example}\label{exmp:frieze}
Because of conditions $1$-$3$ in Definition \ref{def:frieze}, it suffices to know $f(i,j)$ where $i+1 < j$. 

The following non-trivial edge weights satisfy all Ptolemy relations between them and thus give a frieze on a pentagon:
$f(0,2)=f(0,3) = 1, f(1,3) = f(2,4)=2$ and $f(1,4)=3$. For example, the arcs in the blue quadrilateral in Figure \ref{fig:Ex1} must satisfy $f(1,4) \cdot f(0,2)= f(0,1) \cdot f(2,4) + f(0,4) \cdot f(1,2)$ and one can verify that this holds for the specified weights. In fact, once we specify $f(0,2)$ and $f(0,3)$, the other values are forced, and similarly for any other pair of arcs which triangulate the pentagon. 

\begin{figure}
\centering
\begin{tikzpicture}[scale = 1]
\draw[thick,blue] (90:2\R) -- (90+72:2\R);
\draw (90+72:2\R) -- (90+2*72:2\R);
\draw (90+2*72:2\R) -- (90+3*72:2\R);
\draw[thick,blue] (90+3*72:2\R)-- (90+4*72:2\R);
\draw[thick,blue] (90+4*72:2\R) -- (90:2\R);
\draw (90:2\R) -- (90+2*72:2\R);
\draw[thick,blue] (90+72:2\R) -- (90+3*72:2\R);
\draw (90+72:2\R) -- (90-72:2\R);
\draw (90+72*2:2\R) -- (90+72*4:2\R);
\draw (90:2\R) -- (90+3*72:2\R);
\node[above] at (90:2\R) {$0$};
\node[right] at (90+4*72:2\R) {$1$};
\node[below] at (90+2*72:2\R) {$3$};
\node[below] at (90+3*72:2\R) {$2$};
\node[left] at (90+72:2\R) {$4$};
\end{tikzpicture}
\caption{Illustrating the diagonals in a pentagon. Given two intersecting diagonals, such as $(0,2)$ and $(1,4)$, it suffices to consider the quadrilateral with vertices $0,1,2,4$.}\label{fig:Ex1}
\end{figure}

\end{example}

The next example will be a building block for our more general \emph{friezes from dissections}.  Let $U_k(x)$ denote the \emph{normalized Chebyshev polynomials of the second kind}, determined by initial conditions $U_{-1}(x) = 0, U_0(x) = 1$ and the recurrence,\[
U_k(x)= U_1(x)U_{k-1}(x) - U_{k-2}(x),
\]
and let $\lambda_p:= 2\cos(\pi/p)$.  Then, the \emph{Euclidean frieze}, $\mathcal{L}_p$ on a $p$-gon is given by setting $\mathcal{L}_p(i,j) = U_{\vert j-i \vert - 1}(\lambda_p)$.
For example, if $p = 5$, we define $\mathcal{L}_p(i,i+2) = 2\cos(\pi/5) = \frac{1+\sqrt{5}}{2}$ for all $0 \leq i \leq 4$. Including this with the boundary conditions,  $\mathcal{L}_p(i,i) =0$ and $\mathcal{L}_p(i,i+1) = 1$ completely defines the frieze on the pentagon. 




\subsection{Types of Friezes}

Recall that a \emph{dissection} (also called a  \emph{partial triangulation}) of a polygon $P$ is a set of pairwise non-crossing arcs on $P$.  A \emph{$p$-angulation} is a dissection that divides $P$ into $p$-gons.  Given a dissection $\D$ of a polygon $P$, Holm and J{\o}rgensen define a frieze $f_{\D}$ on $P$ which restricts to the Euclidean frieze $\mathcal{L}_{p_i}$ whenever evaluated on a pair of vertices which lie on the same $p_i$-gon \cite{HJ}. In particular, if $(i,j)$ is an arc in $\D$, $f_{\D}(i,j) = 1$ since this arc is a boundary arc of multiple subpolygons. If $(i,j)$ crosses at least one arc in $\D$, we can determine $f_\D(i,j)$ by iteratively using Condition 4 from Definition \ref{def:frieze} at each intersection of $(i,j)$ and $\D$. We call $f_{\D}$ the \emph{frieze from dissection $\D$}.

For instance, Example \ref{exmp:frieze} is a frieze from a dissection as it arises from the dissection using arcs $(0,2)$ and $(0,3)$. This dissection is in fact a triangulation, and thus the entries of the frieze are all in $\mathbb{Z}$. We give an example of a frieze from a dissection which is not a triangulation.

\begin{example}\label{ex:DissectionFrieze}

\begin{figure}
\centering
\begin{tikzpicture}[scale = 1]
\draw (90:2\R) -- (90+72:2\R)-- (90+2*72:2\R) -- (90+3*72:2\R)-- (90+4*72:2\R) -- (90:2\R);
\draw (90:2\R) -- (90+2*72:2\R);
\node[above] at (90:2\R) {$0$};
\node[right] at (90+4*72:2\R) {$1$};
\node[below] at (90+2*72:2\R) {$3$};
\node[below] at (90+3*72:2\R) {$2$};
\node[left] at (90+72:2\R) {$4$};
\end{tikzpicture}
\caption{A dissection of a pentagon}\label{fig:Ex2}
\end{figure}

Consider the dissection of a pentagon in Figure \ref{fig:Ex2}.  Since $(0,2)$ and $(1,3)$ are diagonals of a sub-quadrilateral, we set $f_\D(0,2) = f_\D(1,3) = U_1(\lambda_4) = 2\cos(\pi/4) = \sqrt{2}$. We also know $f_\D(0,3) = 1$ since $(0,3) \in \D$. We can compute $f_\D(2,4)$ by resolving the intersection of $(2,4)$ and $(0,3)$, \[
f_\D(2,4)  = f_\D(2,4)\cdot f_\D(0,3) = f_\D(0,4) \cdot f_\D(2,3) + f_\D(0,2)\cdot  f_\D(3,4)  = 1\cdot 1+ \sqrt{2}\cdot 1  = 1+\sqrt{2}.
\]
A similar calculation finds $f_\D(1,4) = 1+\sqrt{2}$.
\end{example}

We will focus on friezes from dissections that divide a polygon into triangles and quadrilaterals. Accordingly, our corresponding friezes will have values in $\Z[\sqrt{2}]$. 
We define several other types of friezes over $\Z[\sqrt{2}]$ and compare them to friezes from dissections. We first give two descriptions of friezes over $\mathbb{Z}[\sqrt{2}]$ with elementary properties.

\begin{definition}\label{def:types1}
Let $f$ be a frieze on an $n$-gon $P$ over $\mathbb{Z}[\sqrt{2}]$.
\begin{enumerate}
    \item We say $f$ is a $\mathbb{Z}[\sqrt{2}]_{\geq 1}$ frieze if, for all $0 \leq i < j \leq n-1$, $f(i,j) \geq 1$. 
    \item We say $f$ is a $\mathbb{Z}_{\geq 0}[\sqrt{2}]$ frieze if, for all $0 \leq i < j \leq n-1$, $f(i,j) = a+b\sqrt{2}$ and $a,b \in \mathbb{Z}_{\geq 0}$. By the definition of a frieze, we cannot have $a=b=0$.
\end{enumerate}
\end{definition}

The first frieze defined in Definition \ref{def:types1} could also be called a \emph{super-unital frieze}. A related type of frieze is a \emph{unitary frieze}; these were studied in connection to cluster algebras in \cite{GunawanSchiffler}. Unitary friezes will play a leading role in the remainder of the article

\begin{definition}\label{def:UnitaryFrieze}
Let $f$ be a frieze on a polygon $P$ over an integral domain $R$. We say that $f$ is \emph{unitary} if there exists a triangulation $T = \{\tau_1,\ldots,\tau_m\}$ of $P$ such $f(\tau_i) \in R^\times$ for all $1 \leq i \leq m$. In this case, we will also refer to $T$ as a \emph{unitary triangulation}. 
\end{definition}

Recall the norm $N(\cdot)$ of an element $a+b\sqrt{2}  \in \Z[\sqrt{2}]$ is given by $N(a+b\sqrt{2}) = \vert a^2 - 2b^2 \vert$. The units in $\Z[\sqrt{2}]$ (that is, elements with norm 1) are exactly those elements of the form $(1+\sqrt{2})^m$ for $m \in \Z$. Let $\ell_m = (1+\sqrt{2})^m$.

Note that the frieze in Example \ref{ex:DissectionFrieze} is a $\mathbb{Z}[\sqrt{2}]_{\geq 1}$ frieze and a $\mathbb{Z}_{\geq 0}[\sqrt{2}]$ frieze. Moreover, this frieze is unitary, with $(2,4)$ and $(1,4)$ forming a triangulation with unit weights under $f_\D$. 

We now describe relationships amongst the defined types of friezes. For convenience, in our examples for non-containments, we will give the values of a frieze in the form of a frieze pattern. 

\begin{definition}
Given $n \geq 3$, a (finite) \emph{frieze pattern} of width $n$ over a ring $R$ is an array of $n$ rows consisting of entries $m_{i,j}$ with $i,j \in \mathbb{Z}, 0 \leq j-i \leq n-1$ such that
\begin{itemize}
    \item $m_{i,j} \in R$,
    \item the $k$-th row consists of entries $m_{i,i+k-1}$,
    \item $m_{i,i} = m_{i, i+n-1} = 0$ and $m_{i,i+1} = m_{i,i+n-2} = 1$ for all $i \in \mathbb{Z}$,
    \item every other row is shifted so that we have diamonds $\begin{array}{ccc} & m_{i,j} & \\ m_{i-1,j}&& m_{i,j+1} \\ & m_{i-1,j+1} & \end{array}$, and
    \item each diamond satisfies the \emph{diamond condition}: $m_{i-1,j} m_{i,j+1} - m_{i,j} m_{i-1,j+1} = 1$.
\end{itemize}
\end{definition}

Given a frieze $f$ on an $n$-gon, one can define a frieze pattern of width $n+1$ by setting $m_{i,j} = f(\overline{i},\overline{j})$ for $j-i < n+1$ where $\overline{i} = i \pmod{n}$ and $m_{i,i+n} = m_{i,i} = 0$. Given a frieze pattern with $n$ rows, it must be that the rows are $(n-3)$-periodic \cite{CCFrieze}, so one can equivalently define a frieze given a frieze pattern. 

\begin{prop} \label{prop:Containments}
\begin{enumerate}
    \item The set of friezes from dissections is strictly contained in the set of $\mathbb{Z}_{\geq 0}[\sqrt{2}]$ friezes.
    \item
    The set of $\mathbb{Z}_{\geq 0}[\sqrt{2}]$ friezes is strictly contained in the set of  $\mathbb{Z}[\sqrt{2}]_{\geq 1}$ friezes
    \item The set of unitary friezes is incomparable with the sets of friezes from dissections, $\mathbb{Z}_{\geq 0}[\sqrt{2}]$ friezes, and $\mathbb{Z}[\sqrt{2}]_{\geq 1}$ friezes.
\end{enumerate}
\end{prop}

\begin{proof}
1) The fact that every frieze from a dissection is also a $\mathbb{Z}_{\geq 0}[\sqrt{2}]$ frieze is a consequence of the concept of \emph{traditionally-weighted matchings}, a combinatorial interpretation of entries of a frieze pattern from a dissection given in \cite{banaian2021periodic}. To see that this containment is strict, consider the following frieze pattern which gives the values of a frieze on an octagon.

\begin{center}
\scalebox{0.8}{
$\begin{array}{cccccccccccccccccccccc}
 &0&&0&&0&&0&&0 &&0 &&\\
 & &1&&1&&1&&1&&1&\\
 &\sqrt{2} && 1 + \sqrt{2} && \sqrt{2}&&1 + \sqrt{2} &&\sqrt{2} && 1 + \sqrt{2} && \\
   &&1 + \sqrt{2}&& 1 + \sqrt{2} &&1 + \sqrt{2} &&1 + \sqrt{2}&& 1 + \sqrt{2} &\\
 &2 + \sqrt{2} && 2 &&2 +\sqrt{2} && 2 &&2+ \sqrt{2}&& 2 && \\
 &&1 + \sqrt{2}&& 1 + \sqrt{2} &&1 + \sqrt{2} &&1 + \sqrt{2}&&1 + \sqrt{2} &\\
  &\sqrt{2} && 1 + \sqrt{2} && \sqrt{2}&& 1 + \sqrt{2} &&\sqrt{2} && 1 + \sqrt{2} && \\
 & &1&&1&&1&&1&&1&\\
 &0&&0&&0&&0&&0&&0 && \\
  \end{array}$}
 \end{center}

Since there are no entries 1 in the middle five rows of the frieze pattern, there is no non-boundary arc $\tau$ in the octagon such that $f_{\mathcal{D}}(\tau) = 1$. This means if this frieze came from a dissection, there would be no arcs in the dissection. However, this is not the Euclidean frieze $\mathcal{L}_8$ since it is not 1-periodic and the value $2\cos(\pi/8)$ does not appear in the first row.Therefore, the frieze given by this frieze pattern cannot come from a dissection. 
  
2) It is clear that every $\mathbb{Z}_{\geq 0}[\sqrt{2}]$ frieze is a $\mathbb{Z}[\sqrt{2}]_{\geq 1}$ frieze. 
To see that this containment is strict, consider the following frieze pattern giving the values of a frieze on a hexagon.

 \begin{center}
\scalebox{0.9}{$\begin{array}{cccccccccccccccccccccc}
 &0&&0&&0&&0&&0&&\\
 & &1&&1&&1&&1&&1&\\
 &1 + \sqrt{2} &&\sqrt{2}&&3 - \sqrt{2}&&1 + \sqrt{2} &&\sqrt{2}&&\\
   &&1 + \sqrt{2}&&-3 + 3\sqrt{2}&&2\sqrt{2}&&1 + \sqrt{2}&&-3 + 3\sqrt{2}&\\
 &1 + \sqrt{2} &&\sqrt{2}&&3 - \sqrt{2}&&1 + \sqrt{2} &&\sqrt{2}&&\\
 & &1&&1&&1&&1&&1&\\
 &0&&0&&0&&0&&0&&\\
  \end{array}$}
   \end{center}

3) To see that there are unitary friezes that are not $\mathbb{Z}[\sqrt{2}]_{\geq 1}$ friezes, consider the frieze $f$ on a quadrilateral with vertices $0,1,2,3$ which has $f(0,2) = -1+\sqrt{2}$ and $f(1,3) = 2+\sqrt{2}$. Since the diagonal $(0,2)$ triangulates the quadrilateral and $-1+\sqrt{2} \in \mathbb{Z}[\sqrt{2}]^\times$, $f$ is a unitary frieze but $-1+\sqrt{2} < 1$ so $f$ is not a $\mathbb{Z}[\sqrt{2}]_{\geq 1}$ frieze, implying it is also not a $\mathbb{Z}_{\geq 0}[\sqrt{2}]$ frieze nor a frieze from a dissection. 

The Euclidean frieze $\mathcal{L}_4$ on a quadrilateral is an example of a frieze from a dissection that is not unitary.
\end{proof}

\begin{remark}
Coxeter and Conway's Theorem in \cite{CCFrieze} shows that all of the types of friezes considered here are equivalent when working over $\mathbb{Z}$. That is, all friezes on a polygon with entries in $\mathbb{Z}_{\geq 0}$ are unitary; the arcs with unit weight comprise a triangulation of the polygon. 
\end{remark}

\section{Tower Dissections give Unitary Friezes}\label{sec:Towers}

In this section, we investigate friezes which are both from a dissection and unitary. We begin with a set of dissections that never give a unitary frieze.  Recall a $4$-angulation of a polygon $P$ is a dissection that divides $P$ into quadrilaterals; necessarily, $P$ must be a $2n$-gon for $n > 1$. 

\begin{lemma}\label{lem:friezeFrom4angulationNotUnitary}
A frieze from a $4$-angulation of a polygon is never unitary.
\end{lemma}

\begin{proof}
Let $\D$ be a $4$-angulation of a polygon $P$ with vertices $0,\ldots,n-1$. By combining observations from the proof of Proposition 11 in \cite{andritsch2020note} with the correspondence between friezes and frieze patterns, one sees that if $k$ is odd, then $f_\D(i,i+k+1) = b\sqrt{2}$, with $b \in \mathbb{Z}_{\geq 1}$ and if $k$ is even, $f_\D(i,i+k+1) \in \mathbb{Z}_{\geq 1}$. Recall that the units of $\mathbb{Z}[\sqrt{2}]$ are of the form $(1+\sqrt{2})^n$ for $n \in \mathbb{Z}$; thus, the only possible unit in a frieze from a $4$-angulation is 1. Holm and J{\o}rgensen show that the only time that $f_{\D}(\gamma) = 1$ is if $\gamma$ is an arc in $\D$ \cite{HJ}. However, the set of arcs from $\D$ will not be large enough to triangulate $P$. 
\end{proof}


Next, we introduce a family of dissections, called \emph{towers}, which produce unitary friezes. Informally, a tower is a dissection into a straight string of quadrilaterals with a triangle on one end.

\begin{definition}
An \emph{$n$-tower} is a dissection of a $(2n+3)$-gon which consists of one triangle and $n$-quadrilaterals such that no vertex is incident to more than two subpolygons.  We call the portion of the tower excluding the triangle a \emph{stack}. The \emph{roof point} is the unique vertex of the tower which is only incident to the only triangle in the tower. 
A \emph{tower arc} is one which is contained in the tower and has one endpoint at the roof point; see the dashed lines below.
\begin{center}
\begin{tikzpicture}
\draw (1,0) rectangle (5,1);
\draw(3,0.5) node {$\cdots$};
\draw (1,0) -- (0.5,0.5) -- (1,1);
\draw (2,0) -- (2,1);
\draw (4,0) -- (4,1);
\draw[dashed] (0.5,0.5) -- (5,0);
\draw[dashed] (0.5,0.5) -- (5,1);
\end{tikzpicture}
\end{center}

\end{definition}

We allow $n = 0$ in the definition of a tower where a $0$-tower is simply a triangle.  To show how towers yield unitary friezes, we first define two families of arcs in a stack. 

\begin{definition}\label{def:SnDn}
 Given a stack, consider an arc $\sigma_n$ of the form $(i,i+n)$ which goes between two vertices on the same side of the stack, so that, when $n > 0$, $\sigma_n$ crosses $n-1$ arcs from $\D$. We define $s_n$ to be the frieze value $s_n := f_\D(\sigma_n)$. In particular, $\sigma_0$ is a trivial arc with the same start and endpoint, so we have $s_0 = 0$.
 
 Similarly, for $n \geq 1$, consider an arc $\delta_n$  between two vertices on the opposite side of the stack which crosses $n-1$ arcs in the stack. We define $d_n$ to be the frieze value $d_n := f_\D(\delta_n)$. We let $\delta_0$ be an arc from $\D$ of the stack, so that $d_0 = 1$. 
\end{definition}

For example, below the thick arc would have weight $d_4$ and the dashed arc would have weight $s_3$. The first four values of $s_n$, starting at $n = 0$, are $0,1,2\sqrt{2},7$, and the first few values of $d_n$ are $1,\sqrt{2}, 3, 5\sqrt{2}$. The Ptolemy relation implies a simple recurrence amongst these quantities.

\begin{center}
\begin{tikzpicture}
\draw (0,0) -- (4,0) -- (4,1) -- (0,1) -- (0,0);
\draw(1,0) -- (1,1);
\draw(2,0) -- (2,1);
\draw(3,0) -- (3,1);
\draw[thick] (0,0) -- (4,1);
\draw[dashed] (1,0) to[out = 45, in = 135, looseness = 0.5] (4,0);
\end{tikzpicture}
\end{center}

\begin{lemma}\label{lem:RecurrenceSnDn}
The quantities $s_n,d_n$ satisfy the initial conditions $s_0 = 0,d_0=1, s_1 = 1$ and $d_1 = \sqrt{2}$ and for $n \geq 2$, 
\[
s_n = \sqrt{2}s_{n-1} + d_{n-1} \qquad d_{n-1} = \sqrt{2}d_{n-1} + s_{n-1}.
\]
\end{lemma}

\begin{proof}
Suppose that we are working in a stack of size $m$ for $m>>n$. Label the vertices of this stack $0,1,\ldots,2m+1$ in such a way that the arcs of the dissection are of the form $(i,2m+1-i)$ for $1 \leq i \leq m-1$. See Figure \ref{fig:ProofLemRecurrenceSnDn} for an illustration.

\begin{figure}
\centering
\begin{tikzpicture}
\draw(0,2) -- (8,2) -- (8,0) -- (0,0) -- (0,2);
\draw (2,2) -- (2,0);
\draw (6,2) -- (6,0);
\node[] at (4,1){$\cdots$};
\node[above] at (0,2){$2m+1$};
\node[above] at (2,2){$2m$};
\node[above] at (6,2){$m+2$};
\node[above] at (8,2){$m+1$};
\node[below] at (8,0){$m$};
\node[below] at (6,0){$m-1$};
\node[below] at (2,0){1};
\node[below] at (0,0){0};
\end{tikzpicture}
\caption{Illustrating the indexing used in the proof of Lemma \ref{lem:RecurrenceSnDn}}\label{fig:ProofLemRecurrenceSnDn}
\end{figure}

The initial conditions are clear. We focus on the arc $(0,n)$, since we know $f_\D(0,n) = s_n$, and apply the Ptolemy relation to its intersection with  the dissection arc $(n-1,2m+1-(n-1))$, \begin{align*}
f_{\D}(0,n) f_{\D}(n-1,2m+1-(n-1)) &= f_{\D}(0,n-1)f_{\D}(n,2m+1-(n-1)) \\&+ f_{\D}(0,2m+1-(n-1))f_{\D}(n-1,n) \end{align*}
By definition $f_{\D}(0,n-1) = s_{n-1}$ and $ f_{\D}(0,2m+1-(n-1)) = d_{n-1}$. Moreover, the arc $(n-1,n)$ is on the boundary, and $(n,2m+1-(n-1))$ is a diagonal in a quadrilateral in the dissection. Therefore we can conclude that $s_n = f_\D(\sigma_n) = \sqrt{2}s_{n-1} + d_{n-1}$. The other recurrence can be proven similarly.
\end{proof}

The recurrence in Lemma \ref{lem:RecurrenceSnDn} will allow us to show that tower arcs have unit weight. Recall we set $\ell_n = (1+\sqrt{2})^n$.

\begin{lemma}\label{lem:SnPlusDnIsUnit}
The quantities $s_n$ and $d_n$ satisfy the recurrences \[
s_n + d_n =\ell_n.
\]
\end{lemma}

\begin{proof}
We induct on $n$. The claim is true for $n=0$ since $s_0 = 0$ and $d_0 = 1$. Suppose we have shown this is true for $n-1$. From Lemma \ref{lem:RecurrenceSnDn}, we can expand $s_n + d_n$ as \[
s_n + d_n = (\sqrt{2}s_{n-1} + d_{n-1}) + (s_{n-1} + \sqrt{2}d_{n-1}) = (1+\sqrt{2})(s_{n-1}+d_{n-1}) = (1+\sqrt{2})^n
\]
where the last equality holds by our inductive hypothesis. 
\end{proof}

\begin{cor}\label{cor:TowerArcsUnits}
A tower arc $\gamma$ which passes through an $n$ tower has weight $f_{\mathcal{D}} = \ell_n$.
\end{cor}

\begin{proof}
If $n = 0$, then $\gamma \in \D$ and $f_{\D}(\gamma) = 1$ by definition. If $n > 0$, then we consider the intersection between $\gamma$ and the arc from $\D$ in the tower which borders both the triangle and the first quadrilateral in the tower. Applying the Ptolemy relation to this intersection and using Lemma \ref{lem:SnPlusDnIsUnit} yields the result.

\begin{center}
\begin{tikzpicture}
\draw (1,0) rectangle (5,1);
\draw(3,0.5) node {$\cdots$};
\draw (1,0) -- (0.5,0.5) -- (1,1);
\draw (2,0) -- (2,1);
\draw (4,0) -- (4,1);
\draw[line width=0.35mm, blue] (0.5,0.5) -- (5,0);
\draw[line width=0.35mm, blue] (1,0) -- (1,1);
\end{tikzpicture}
\end{center}
\end{proof}

Now that we know that tower arcs have unit weight, we use these to build unitary triangulations of dissections which are the result of combining multiple towers. We say a dissection $\D$ of $P$ is a gluing of towers if we can decompose $P$ into a set of subpolygons such that $\D$ restricted to each subpolygon is a tower.

\begin{thm}\label{thm:Towers}
Let $\D$, a dissection on polygon $P$, which can be decomposed into a set of towers. Then, the frieze $f_{\D}$ is unitary.
\end{thm}

\begin{proof}
We first show that a single tower yields a unitary frieze.  Label the vertices of the tower as below. 

\begin{center}
\begin{tikzpicture}
\draw (11,-1) rectangle (12,0);
\draw(11.5,-1.3) node {$\vdots$};
\draw (11,-2) rectangle (12,-3);
\draw (11,0) -- (11.5,0.5) -- (12,0);
\node[above] at (11.5,0.5){$0$};
\node[right] at (12,0) {$2m+2$};
\node[right] at (12,-3){$m+2$};
\node[left] at (11,-3){$m+1$};
\node[left] at (11,0){$1$};
\node[left] at (11,-1){$2$};
\node[left] at (11,-2) {$m$};
\node[right] at (12,-1){$2m+1$};
\node[right] at (12,-2){$m+3$};
\end{tikzpicture}
\end{center}

By Corollary \ref{cor:TowerArcsUnits}, every arc of the form $(0,i)$ for $i \neq 0$ has unit weight under $f_\D$. If $1 < i < 2m+2$,  the arc $(0,i)$ is not a boundary arc. Moreover, since no pair of distinct arcs from $\{(0,i): 1 < i < 2m+2\}$ will cross and this set is size $2m$, we see that this set is a triangulation of the tower.

Now, suppose our dissection $\D$ is composed of several towers, glued together along edges $\tau_1,\ldots,\tau_\ell$.  We can form a unitary triangulation of $\D$ by triangulating each tower, as described above,  and then including the glued edges $\tau_1,\ldots,\tau_\ell$.  Since $f_\D(\tau_i) = 1$ by definition, this triangulation is unitary. 
\end{proof}

Since we consider a triangle as a 0-tower, Theorem \ref{thm:Towers} recovers the fact that we can find a unitary triangulation for a frieze from a triangulation. When working with frieze from a triangulation $T$, the only choice for a unitary triangulation is the original triangulation since any arc not in the $T$ must have weight strictly larger than 1. This is true since any arc not in the $T$ will cross at least one arc in $T$, and so by the Ptolemy relation, the weight of the arc not in $T$  must be a sum of two different non-negative integers.

Gunawan and Schiffler show that, in the case of a friezes over $\mathbb{Z}$ on a polygon, there is a bijection between unitary friezes and triangulations of the polygon. The situation is different in our setting. For example, here we give a dissection of a 10-gon which admits two unitary friezes.  

\begin{example}

\begin{figure}
\centering
\begin{tikzpicture}
\draw (0,0) -- (2,0) -- (2.5,.5) -- (2,1) -- (1,1) -- (1,2) -- (0.5,2.5) -- (0,2) -- (0,0);
\draw (0,1) -- (1,1) -- (1,0);
\draw (2,0) -- (2,1);
\draw (0,2) -- (1,2);
\node[above] at (0.5,2.5){0};
\node[left] at (0,2){1};
\node[left] at (0,1){2};
\node[left] at (0,0){3};
\node[below] at (1,0){4};
\node[below] at (2,0){5};
\node[right] at (2.5,.5){6};
\node[above] at (2,1){7};
\node[above, xshift = 5pt] at (1,1){8};
\node[right] at (1,2){9};
\end{tikzpicture}
\caption{An example of a dissection of a 10-gon which is the result of gluing two towers.}\label{fig:Ex10gon}
\end{figure}

We consider the dissection of a 10-gon provided in Figure \ref{fig:Ex10gon}. Note that there are two ways we can decompose this dissection into towers. This gives two options for unitary triangulations. One can show that these are in fact the only two options. For example, there are no arcs of the form $(1,v)$ which have unit weight, so we need to include the arc $(0,2)$ in any unitary triangulation. Conversely, since $f_{\D}(2,4) = \sqrt{2}$, we cannot include $(2,4)$ in any unitary triangulation, and as a result we must have at least one arc incident to vertex 3.
\begin{itemize}
    \item $\{(0,2), (0,3), (0,4),(0,8),(4,8),(6,4),(6,8)\}$
     \item $\{(0,2), (0,3), (2,8),(6,2), (6,3),(6,4),(6,8)\}$
\end{itemize}
\end{example}

We conjecture that Theorem \ref{thm:Towers} can be made stronger and that dissections from gluings of towers are the only types of dissections into triangles and quadrilaterals admitting a unitary triangulation. 

\begin{conjecture}\label{conj:UnitaryConj}
Let $\D$ be a dissection of a polygon into triangles and quadrilaterals.  Then,  the frieze $f_\D$ is unitary if and only if $\D$ can be decomposed into a set of towers.  
\end{conjecture}

In the remainder of this article, we make progress towards Conjecture \ref{conj:UnitaryConj} by verifying it for a couple families of dissections.

\begin{remark}
A consequence of Theorem 1 in \cite{cuntz2023frieze} is that there are infinitely many friezes on an $n$-gon with values in $\mathbb{Z}[\sqrt{2}]$. While we work in the finite subset of friezes which arise from a dissection, it would be also interesting to study the set of all unitary friezes on an $n$-gon with values in $\mathbb{Z}[\sqrt{2}]$. 
\end{remark}

\subsection{Connection to Continued Fractions}

We also note a relationship between the quantities $s_n$ and $d_n$ and the \emph{continued fraction} expansion of $\sqrt{2}$. By continued fraction, we mean an expression \[
[t_0,\ldots,t_n] = t_0 +\cfrac{1}{t_1 +\cfrac{1}{t_2 +\cfrac{1}{
   \ddots + \cfrac{1}{t_n}}}},
   \]
where the $t_i \in \mathbb{Z}_{>0}$. Every rational number can be written as a continued fraction with finitely many $t_i$; this expansion can be calculated using the Euclidean algorithm. Moreover, the numerator and denominator we get from calculating a continued fraction can always be shown to be already relatively prime so that they are in lowest term. We can also define an \emph{infinite continued fraction} as a limit of finite continued fractions; every irrational number can be written as an infinite continued fraction. For example, $\sqrt{2} = [1,2,2,2,\ldots]$. 

Let $\frac{a_n}{b_n} = [1,2,\ldots,2]$ where there are $n-1$ entries of $2$. Set $a_0 = 0$ and $b_0 = 1$. Then, we have the following relationship between the sequences $\{a_n\}_n, \{b_n\}_n, \{s_n\}_n,$ and $\{d_n\}_n$. 

\begin{prop}\label{prop:relateSnDnContFrac}
If $n \geq 0$ is even, then \[
s_n = b_n \sqrt{2} \qquad d_n = a_n,
\]
and if $n > 0$ is odd, then \[
s_n = a_n \qquad d_n = b_n\sqrt{2}.
\]
\end{prop}

\begin{proof}
We begin by claiming $a_n = a_{n-1} + 2b_{n-1}$ and $b_n = a_{n-1} + b_{n-1}$ when $n \geq 2$. The case for $n = 2$ can be checked directly - we have $a_1 = b_1 = 1, a_2 = 3$ and $b_2 = 2$. Now suppose we have shown the claim for $n-1$. Since $n \geq 2$, the last term in the continued fraction is 2. Thus, it is well-known we can write \[
\frac{a_n}{b_n} = \frac{2 a_{n-1} + a_{n-2}}{2b_{n-1} + b_{n-2}};
\]
for further explanation, see for example \cite{frame1949continued}. Now, by induction we have \[
a_n = a_{n-1} + (a_{n-2} + 2b_{n-2}) + a_{n-2} = a_{n-1} + 2b_{n-1},
\]
where we use our inductive hypothesis first at $a_{n-1}$ and then regather terms to form $b_{n-1}$. One can show $b_n = a_{n-1} + b_{n-1}$ similarly. 

Now that we have this identity, we again use induction to prove the statement of the proposition. The $n = 0$ case is immediate. Suppose we have shown the claim for the $n-1$ case where $n-1$ is even. Then, $n$ is odd, and by Lemma \ref{lem:RecurrenceSnDn} and the above recurrence on $a_n$ and $b_n$, we have \[
s_n = \sqrt{2}s_{n-1} + d_{n-1} = \sqrt{2}(b_{n-1} \sqrt{2}) + a_{n-1} = a_n
\]
and
\[
d_n = \sqrt{2}d_{n-1} + s_{n-1} = \sqrt{2}a_{n-1} + \sqrt{2}b_{n-1} = \sqrt{2} b_n.
\]
The even $n$ case follows identically by swapping the roles of $s_n$ and $d_n$. 
\end{proof}

Since by definition $\frac{a_n}{b_n} \to \sqrt{2}$, we see that the quantities $s_n$ and $d_n$ approach the same number as $n$ gets large. 

The sequence $\{b_n\}_n$ is exactly the sequence of  \emph{Pell numbers}. In Remark \ref{rem:Pell} we use the Pell numbers to exhibit another family of arcs which have unit weight.

\section{Separated Dissections}\label{sec:Separated}


We call a disesction of a polygon $P$ into triangles and quadrilaterals \emph{separated} if (1) there is only one arc, $\tau_a$, in $\D$ which borders both a triangle and a quadrilateral and (2) the quadrilateral incident to this arc has sides $\tau_a,\tau_b,\tau_c,\tau_d$ in clockwise order such that $\tau_b$ and $\tau_d$ are boundary edges of $P$. See Figure \ref{fig:ExSeparated}.

\begin{figure}
\centering
\begin{tikzpicture}[scale = 1]
\draw (0,-.25) -- (1,0.5) -- (6,0.5) -- (6,-1) -- (1,-1) -- (0,-.25);
\draw (1,0.5) -- (1,-1);
\draw (1,0.5) -- (2,-1);
\draw (2,-1) to node[right, xshift = -3pt]{$\tau_a$} (2,0.5);
\draw (3.3,-1) to node[left,xshift = 3pt]{$\tau_c$} (3.3,0.5);
\draw (4.65,-1) -- (4.65,2) -- (3.3,2) -- (3.3,0.5);
\draw(4.65,-1) -- (4.65,-2.5) -- (6,-2.5) -- (6,-1);
\node[below,yshift = 3pt] at (2.65,0.5){$\tau_b$};
\node[above,yshift = -3pt] at (2.65,-1){$\tau_d$};
\end{tikzpicture}
\caption{An example of a separated dissection}\label{fig:ExSeparated}
\end{figure}

\begin{prop}\label{prop:Separated}
Conjecture \ref{conj:UnitaryConj} holds when the dissection is separated. 
\end{prop}

\begin{proof}
By Theorem \ref{thm:Towers}, we just need to show that if a separated dissection $\mathcal{D}$ is not a gluing of towers, then the frieze $f_\D$ is not unitary. Since the dissection is separated, this means that either there are no triangles in $\D$ or there is at least one triangle but the set of quadrilaterals in the dissection does not form a stack. We know the claim is true in the former case by Lemma \ref{lem:friezeFrom4angulationNotUnitary}, so we assume we are in the latter case.

Let $P$ be an $n$-gon with a separated dissection $\D$ which cannot be decomposed into towers. Let $Q_0$ be the unique quadrilateral in $\D$ which shares vertices with a triangle, and let this triangle be $\Delta_0$. Label the vertices of $P$ in clockwise order such that the vertices of $Q_0$ are $1,m,m+1,n$ for $m < n$ as in Figure \ref{fig:SeparatedProof}. Then, the vertex of $\Delta_0$ which is not incident to $Q_0$ is $k$ with $1 < k < m$. By our assumption, there exists a quadrilateral $Q_1$ in $\D$ which is not in the stack containing $Q_0$ but which does share vertices with a quadrilateral in this stack. Of the two vertices of $Q_1$ which are not on the stack, let $v$ be the one with larger index, as in Figure \ref{fig:SeparatedProof}. In the picture below, the hexagons represent arbitrary polygons with a triangulation; it is possible that one or both of these do not exist so that $k$ could be $2$ or $m-1$. It is also possible that there are 4-angulated polygons glued on all edges of quadrilaterals except $(1,n)$ and $(m,m+1)$. 

\begin{figure}
\centering
\begin{tikzpicture}
\draw (0,0) -- (1,0) -- (1,-3) -- (-1,-3) -- (-1,-2) -- (0,-2) -- (0,0);
\draw (0,-3) -- (0,-2) -- (1,-2);
\draw (0,-1) -- (1,-1);
\draw (0,0) -- (0.5,1) -- (1,0);
\draw (0,0) -- (-1,0) -- (-1.5, 0.75) -- (-1, 1.5) -- (-0.25,1.5) -- (0.5,1);
\draw (1,0) -- (2,0) -- (2.5,0.75) -- (2,1.5) -- (1.25,1.5) -- (0.5,1);
\node[above] at (0.5,1){$k$};
\node[right, yshift = -5pt, xshift = -3pt] at (1,0) {$m$};
\node[right, xshift = -3pt] at (1,-1){$m+1$};
\node[left, yshift = -5pt, xshift = 3pt] at (0,0){$1$};
\node[left, xshift = 3pt] at (0,-1){$n$};
\node[above,yshift = -3pt] at (-1,-2){$v$};
\node[below, yshift = 3pt] at (-1,0){$j$};
\node[left, xshift = 3pt] at (-1.5,0.75){$i$};
\node[] at (-0.5,-2.5){$Q_1$};
\node[] at (0.5,-0.5){$Q_0$};
\node[] at (0.55,0.35){$\Delta_0$};
\end{tikzpicture}
\caption{Illustrating the notation for the proof of Proposition \ref{prop:Separated}}\label{fig:SeparatedProof}
\end{figure}

Assume for sake of contradiction that there exists a unitary triangulation $T$ of $P$.  Notice that for every vertex $1 \leq i \leq n$, we either have at least one arc in $T$ incident to $i$ or we include the arc $(i-1,i+1)$ in $T$. If $m < i < n$, we cannot add the arc $(i-1,i+1)$ because the weight of this arc will be a positive integer multiple of $\sqrt{2}$, so it will not be a unit. This means that there must be at least one arc from $T$ incident to each vertex $i$ for $m < i < n$. 

We apply this observation to vertex $v$. One can check that $f_\D(v,k) = (1+2\sqrt{2})\ell_a$ where $a+1$ is the number of quadrilaterals $(v,k)$ passes through. We see that $N((1+2\sqrt{2})\ell_a) = N(1+2\sqrt{2}) N(\ell_a)= 7$ so $(v,k)$ cannot be in $T$. If the triangulated region only consists of this triangle, we are done. Otherwise, let $i$ be the smallest value such that $(v,i) \in T$. We first assume $i < k$. 

We cannot have $(n,k) \in T$ since this would cross $(v,i)$. Let $j$ be the smallest value such that $(n,j) \in T$. We know such a value exists by the previous discussion, and we know that $j \leq i$ since otherwise $T$ would contain  a pair of intersecting arcs.

Since $j$ is minimal, $(n,j)$ is in the same triangle in the triangulation as the boundary arc $(1,n)$. The third side of this triangle, $(1,j)$, only passes through the triangulated part of $\D$. This means that $f_{\D}(1,j) \in \mathbb{Z}$. Since we need this to be a unit, it must be that  $f_{\D}(1,j) = 1$. Now, since $(1,k)$ and $(n,j)$ cross, we can use the Ptolemy relation we analyze the relation between $f_{\D}(1,k)$ and $f_{\D}(n,j)$,
\[
f_{\D}(n,j) = f_{\D}(n,j) f_{\D}(1,k) = f_{\D}(1,n) f_{\D}(j,k) + f_{\D}(1,j) f_{\D}(k,n) = b + \ell_1
\]
where $b \geq 1$ is a positive integer since the arc $(j,k)$ only crosses triangles. We again see here that $f_{\D}(n,j)$ is not an integer. Therefore, it is impossible to create such a triangulation.

If we had $i > k$ instead, we could again show that such a triangulation is impossible by repeating the above arguments with vertices $m$ and $ m+1$. 
\end{proof}

\section{Type 3 Dissections}\label{sec:Type3}

In this section we verify that Conjecture \ref{conj:UnitaryConj} is true for another family of dissections. 

\begin{definition}
A \emph{type $d$} dissection is a polygon dissected into squares and triangles such that each vertex belongs to no more than $d$ subpolygons. Equivalently, viewing the dissected polygon as a graph, every vertex of the polygon has degree at most $d+1$. 
\end{definition}

\begin{remark}\label{rem:typetwo}
Note that a type 2 dissection would be a stack, a tower, or a gluing of two towers. The fact that a stack cannot be given a unitary triangulation follows from Lemma \ref{lem:friezeFrom4angulationNotUnitary} while a tower or gluing of two towers can be given a unitary triangulation by Theorem \ref{thm:Towers}. Thus, we know our conjecture is true for type 2 dissections.
\end{remark}

We first introduce a useful class of triangles in a triangulation. Let a \emph{basic triangle} in a dissection of a surface be a triangle with exactly two sides along the boundary of the surface. The following fact can be found for example in \cite{BCI}.

\begin{lemma} \label{lem:2Ears}
In any triangulation of an $n$-gon, $n \geq 4$, there are at least two basic triangles. 
\end{lemma} 

\begin{proof}

Let $P$ be an $(n+3)$-gon and let $T$ be a triangulation of $P$. Then, $T$ consists of $n$ arcs and divides $P$ into $n+2$ triangles.

We will count, with multiplicity,  the number of non-boundary sides of triangles in two ways. On the one hand, since there are $n$ arcs in $T$, this number must be $2n$. Let $A_1 \geq 0$ be the number of basic triangles, $A_2 \geq 0$ the number of triangles with one edge along the boundary, and $A_3 \geq 0$ the number of triangles with no edges along the boundary. Then, $n+2 = A_1 + A_2 + A_3$, and \[
2n = A_1 + 2A_2 + 3A_3 \leq A_1 + 2(A_2 + A_3).
\]

We see that this equation will not have any solutions unless $A_1$ is at least 2. 


\end{proof}

We will show Conjecture \ref{conj:UnitaryConj} holds for type 3 dissections by describing an algorithm for building triangulations of a polygon which are unitary with respect to $f_{\mathcal{D}}$ and then showing that $\mathcal{D}$ must be a gluing of towers in order for our algorithm to terminate in step (3). The algorithm will begin with a polygon with dissection and will try to build a unitary triangulation by choosing arcs with unit weight one by one. Viewing the construction of a unitary triangulation chronologically allows us to describe what sort of arcs can and cannot be used. For example, in Lemma \ref{lem:AlgFailsWithPell}, we show that after using a certain type of arc called a ``Pell arc'', our algorithm will have to terminate early and will not produce a unitary triangulation.

\begin{algorithm}\label{TriAlgorithm} The input of our algorithm will be an $n$-gon $P = P_0$ on vertices $\{0,\ldots,n-1\}$ and a dissection $\mathcal{D}$ of $P_0$ into triangles and quadrilaterals. We initialize $T = \emptyset$. Notice we do not change the dissection $\mathcal{D}$ during the algorithm.
\begin{enumerate}
    \item[(0)] Let $i = 0$.
    \item[(1)] If $i = n-3$, then $P_i$ is a triangle and the algorithm terminates. 
    \item[(2)] If $i < n-3$, and there exists a diagonal $\tau_i = \{a,b\}$ which forms a basic triangle in $P_i$ such that $f_{\mathcal{D}}(a,b) \in \Z[\sqrt{2}]^\times$, we add $\tau_i$ to $T$. We form an $(n-i-1)$-gon $P_{i+1}$ by removing the boundary edges of this basic triangle, so that $(a,b)$ is now a boundary edge.  Add 1 to $i$ and return to step 2.
    \item[(3)] If $i < n-3$, and for every diagonal $(a,b)$ forming a basic triangle in $P_i$, $f_{\mathcal{D}}(a,b) \notin \Z[\sqrt{2}]^\times$, then the algorithm terminates. 
\end{enumerate}
\end{algorithm}

If the algorithm terminates in step (1), then the set $T$ is a unitary triangulation of $P$. If the algorithm terminates in step (3), then it is impossible for the partial triangulation $T$ to be completed to a unitary triangulation.

There will often be more than one possible arc that we could add in step (2) of the Triangulation Algorithm, and this choice can affect future choices. Therefore, we must run the Triangulation Algorithm multiple times to exhaustively test whether there exists a unitary triangulation of a polygon $P$. The number of times we would need to run the algorithm is bounded above by the product of the number of triangulations of $P$ and the number of permutations of the arcs in each triangulation; in particular, we only need to run the algorithm a finite number of times.

In Theorem \ref{thm:Towers}, we showed that unitary triangulations always exist with respect to dissections which are the result of gluing towers. Thus, we will show Conjecture \ref{conj:UnitaryConj} is true for type 3 dissections by showing, conversely, that the Triangulation Algorithm will never terminate with a triangulation of $P$ (i.e., in step (3)) if $\mathcal{D}$ is not a gluing of towers. 

\begin{thm}\label{thm:TypeThreeNonTowerFails}
If $\mathcal{D}$ is a type $d$ dissection of polygon $P$, for $d \leq 3$, which is not a gluing of towers, then the Triangulation Algorithm will never produce a unitary triangulation when it begins with $P$ and $\mathcal{D}$. That is, the algorithm will always terminate in step (3).
\end{thm}

We will prove a series of smaller results and then put them together to prove Theorem \ref{thm:TypeThreeNonTowerFails}. We begin by showing that the first arcs which appear in the Triangulation Algorithm must be tower arcs or arcs from $\D$.



\begin{lemma}\label{lem:BasicTriangleAlongBoundary}
Consider a basic triangle with both boundary edges having weight 1. Then, the non-boundary edge of this triangle is either a tower arc, an arc in the dissection, or has non-unit weight. 
\end{lemma}

\begin{proof}

 If $e$ is a diagonal skipping exactly one vertex, $v$, in a polygon $P$, then $f_\D(e) = \sum_{p_i} \lambda_{p_i}$ where we sum over the sizes $p_i$ of the subpolygons that $v$ is incident to (see \cite{HJ} ). Since the units in $\mathbb{Z}[\sqrt{2}]$ are $(1+\sqrt{2})^n$, and a vertex can be incident to at most three subpolygons in a type 3 dissection, the only units we will see are 1 and $1+\sqrt{2}$. The former occurs when $e$ is an arc in the dissection bounding a triangle and the latter occurs when $e$ is a tower arc. 
 
 \begin{center}
\begin{tabular}{c|c}
  \begin{tikzpicture}[scale = 1.5]
  \draw (0,0) -- (1,1) -- (2,0) -- (0,0);
  \node[above] at (1,1){$v$};
  \draw[thick] (2,0) to node[below]{1} (0,0);
  \end{tikzpicture}   &  
  \begin{tikzpicture}
  \draw (0,0) -- (4,0) -- (2,2) -- (0,2) -- (0,0);
  \draw (2,0) -- (2,2);
  \node[above] at (2,2){$v$};
  \draw[thick] (4,0) -- (0,2);
  \node[below] at (1,1.25){$1+\sqrt{2}$};
  \end{tikzpicture}
  \\
\end{tabular}
\end{center}

\end{proof}

It is possible that, in the course of running the triangulation algorithm, arcs other than tower arcs appear. These will be the result of forming basic triangles with arcs which were not boundary arcs in the original polygon. 

\begin{definition}\label{lem:case5arc}
Consider a dissection in which an $i$-tower and a $j$-tower, with $i \geq 1$ and $j \geq 0$, which share one vertex. Moreover, this vertex is also part of a triangle; see Figure \ref{fig:Pell}. We refer to the arc between the roof points of the towers (where we consider the roof point of the $j$-tower to be the vertex not adjacent to the $i$-tower) a \emph{Pell arc}. This is arc $(a,b)$ using notation from Figure \ref{fig:Pell}.

\end{definition}

First, we show that Pell arcs have unit weight. 

\begin{figure}
\centering
\begin{tikzpicture}[scale = 0.7]
 \draw(0,0) -- (1,1) -- (1,-1) -- (0,0);
 \draw (1,1) -- (8,1) -- (8,-1) -- (1,-1);
 \draw (6,-1) -- (8,-3) -- (8,-1);
 \draw (6,-1) -- (6,1);
 \draw (3,1) -- (3,-1);
\draw (8,-1) -- (15,-1) -- (15,-3) -- (8,-3);
\draw (15,-1) -- (16,-2) -- (15,-3);
\draw (10,-1) -- (10,-3);
\draw (13,-1) -- (13,-3);
\node[] at (4.5,0){$\cdots$};
\node[] at (11.5,-2){$\cdots$};
\draw [decorate,
    decoration = {brace}] (1,1.3) --  (8,1.3);
\node[] at (4.5,1.8){$i \geq 1$};
\draw [decorate,
    decoration = {brace}] (15,-3.3) --  (8,-3.3);
\node[] at (11.5, -3.8){$j \geq 0$};
\node[] at (-.3,0){$a$};
\node[] at (16.3,-2){$b$};
\node[below] at (6,-1){$c$}; 
\node[] at (8.2,-0.7){$d$};
\end{tikzpicture}
\caption{The arc between $a$ and $b$ is a Pell arc.}\label{fig:Pell}
\end{figure}

\begin{lemma}
A Pell arc between an $i$-tower and a $j$-tower has weight $\ell_{i+j+1}$.
\end{lemma}

\begin{proof}
We use the notation from Figure \ref{fig:Pell} and resolve the intersection between $(a,b)$ and $(c,d)$. We have \[
f_\D(a,b) = f_\D(a,b) f_\D(c,d) = f_\D(a,c)f_\D(b,d) + f_\D(a,d) f_\D(b,c) = \ell_{i-1}\ell_j + \ell_i f_\D(b,c).
\]
Now, by resolving the intersection between $(b,c)$ and the arc in $\D$ adjacent to $d$ which separates the triangle from the square in the $j$-tower, we find $f_\D(b,c) = 2\ell_j$. Thus, we have \[
f_\D(a,b) = \ell_{i+j-1} + 2\ell_{i+j} = (1 + 2\ell_1)\ell_{i+j-1} = \ell_2\ell_{i+j-1} = \ell_{i+j+1}.
\]

\end{proof}

The name ``Pell arc'' is explained in the following remark, showing a relationship with the Pell numbers. 

\begin{remark}\label{rem:Pell}
 We show here that, if we remove the restriction of a type 3 dissection, then Pell arcs, with general $i$ and $j = 0$ in the notation of Figure \ref{fig:Pell} sit in a larger family of arcs with unit weight. Consider a $(2m)$-gon, with vertices $\{0,\ldots,2m-1\}$ and triangulated with arcs $\{(1,2m-1\} \cup \{(2i, 2m-2i), (2i,2m-2i+1): 0 < 2i < m\} \cup \{(2i-1, 2m - 2i), 2i-1,2m-(2i-1)) : m < 2i+1 < 2m-1\}$. As a set of $2m-3$, non-crossing arcs, this set gives a triangulation. Glue the boundary edge $(0,2m-1)$ of this triangulated polygon onto the one of the boundary edges on the last quadrilateral of a tower as below. This is illustrated in Figure \ref{fig:PellEx}. Then, if the roof point of the tower is $w$, we claim that $f_{\mathcal{D}}(w,m) = \ell_{i+m-1}$. 

 \begin{figure}
 \centering
\begin{tikzpicture}[scale = 0.7]
 \draw(0,0) -- (1,1) -- (1,-1) -- (0,0);
 \draw (1,1) -- (8,1) -- (8,-1) -- (1,-1);
 \draw (6,-1) -- (4.7,-2.5) -- (6,-4) -- (8,-4) -- (9.3,-2.5) -- (8,-1);
 \draw (6,-1) -- (6,1);
 \draw (3,1) -- (3,-1);
 \draw (9.3,-2.5) -- (6,-4) -- (8,-1) -- (4.7,-2.5);
\node[] at (4.5,0){$\cdots$};
\draw [decorate,
    decoration = {brace}] (1,1.3) --  (8,1.3);
\node[] at (4.5,1.8){$i \geq 1$};  
\node[left] at (0,0) {$w$};
\node[right] at (8,-1){$5$};
\node[right] at (9.3,-2.5){$4$};
\node[right,below] at (8,-4){$3$};
\node[left,below] at (6,-4){$2$};
\node[left] at (4.7,-2.5){$1$};
\node[left, xshift = -5pt, yshift = -6pt] at (6,-1){$0$};
\end{tikzpicture}
\caption{As is explained in Remark \ref{rem:Pell}, for any length of stack $i \geq 1$, the arc $(w,3)$ will have unit length.}\label{fig:PellEx}
 \end{figure}
 
 Using $f_{\D}(0,2m-1) = 1$, and the fact that $(0,2m-1)$ and $(w,m)$ cross, we can express $f_{\D}(w,m)$ in terms of arcs in the triangulated subpolygon and tower arcs, \[
f_{\mathcal{D}}(w,m) = \ell_i f_{\mathcal{D}}(0,m) + \ell_{i-1} f_{\mathcal{D}}(2m-1,m).
 \]

 As a consequence of Theorem A in \cite{ccanakcci2018cluster}, $\frac{f_{\mathcal{D}}(0,m)}{f_{\mathcal{D}}(2m-1,m)}$ is given by the continued fraction $[2,2,\ldots,2]$ consisting of $m-1$ 2's. It is well-known that these continued fractions have consecutive Pell numbers $Q_m$ and $Q_{m-1}$ in the numerator and denominator, and $\gcd(Q_m,Q_{m-1}) = 1$. Recall the Pell numbers $Q_k$ are initialized $Q_0 = 0, Q_1 = 1$ and for $k \geq 2$, $Q_k = 2Q_{k-1} + Q_{k-2}$. Therefore, we have that $f_{\mathcal{D}}(w,m) = Q_m \ell_i + Q_{m-1} \ell_{i-1} = \ell_{i-1} ((Q_m + Q_{m-1}) + Q_{m-1} \sqrt{2})$. 
 
 Now, recall our notation from  Proposition \ref{prop:relateSnDnContFrac}, where $a_n$ was the numerator of $[1,2,\ldots,2]$ with $n-1$ entries $2$ and $b_n$ was the denominator of the same continued fraction. We can use the identities from the proof to show $a_n = b_n + b_{n-1}$. Since the $b_n$ are exactly the Pell numbers, we have that $Q_{m-1}\sqrt{2} = b_{m-1}\sqrt{2}$ and $Q_m + Q_{m-1} = b_m + b_{m-1} = a_m$. Thus, from this Proposition, we conclude that $Q_m + Q_{m-1} + Q_{m-1} \sqrt{2} = s_m + d_m = \ell_m$, so that $f_{\mathcal{D}}(w,m) = \ell_{i+m-1}$.

 When $2m > 4$, the dissections described here are not type 3 dissections as they require at least one vertex to be incident to four subpolygons. The existence of these arcs is part of our motivation to focus on type $d$ dissections for $d \leq 3$, given that our current techniques rely on checking a finite number of cases. However, the existence of these arcs does not lead us to believe Conjecture \ref{conj:UnitaryConj} is false. 

\end{remark}

We show next that Pell arcs are the only new type of arcs which can appear in our triangulation algorithm once we have a partial triangulation with arcs from $\D$ and tower arcs.

\begin{lemma}\label{lem:BasicTriangleNotBoundary}
An arc that forms a triangle with two tower arcs or a tower arc and a boundary arc is either a tower arc, a Pell arc, or has non-unit weight. 

\end{lemma}

\begin{proof}
 Since each vertex is incident to at most three subpolygons in the dissection, we can construct a finite list of ways we can form a triangle which has either two tower arcs as edges or a tower arc and a boundary arc as edges. See Appendix \ref{sec:AppA} for a table depicting every such  case; note that some entries denote multiple cases.
 We can prove that, if the third arc is not a tower arc or a Pell arc, then the third arc has non-unit weight by analyzing each such case. The cases that give tower or Pell arcs are entry 2 where both optional shapes are triangles (this gives a tower arc), entry 3 where the optional shape is a quadrilateral (this gives a Pell arc where one tower is a 0-tower), and entry 17 where the optional shape is a triangle (this gives a more general Pell arc). Thus, we prove this Lemma by analyzing all other cases.

Here, we provide a sample calculation to show entry 12 in the table in Appendix \ref{sec:AppA} where the dashed line is deleted does not have unit weight. Let $i,j \geq 1$. Note that the arc $(u,v)$ forms a triangle with the tower arcs $(u,x)$ and $(x,v)$.



\begin{center}
\begin{tikzpicture}[scale = 0.5]
 \draw(0,0) -- (1,1) -- (1,-1) -- (0,0);
 \draw (1,1) -- (8,1) -- (8,-1) -- (1,-1);
 \draw (6,-1) -- (8,-3) -- (8,-1);
 \draw (6,-1) -- (6,1);
 \draw (3,1) -- (3,-1);
\draw (8,-1) -- (15,-1) -- (15,-3) -- (8,-3);
\draw (10,-1) -- (10,-3);
\draw (13,-1) -- (13,-3);
\node[] at (4.5,0){$\cdots$};
\node[] at (11.5,-2){$\cdots$};
\draw [decorate,
    decoration = {brace}] (1,1.3) --  (8,1.3);
\node[] at (4.5,1.8){$i \geq 1$};
\draw [decorate,
    decoration = {brace}] (15,-3.3) --  (8,-3.3);
\node[] at (11.5, -3.8){$j \geq 1$};
\node[] at (-.3,0){$u$};
\node[right] at (15,-3){$v$};
\node[left,yshift = -5pt] at (6,-1){$x$};
\node[right,yshift = 5pt] at (8,-1){$y$};
\end{tikzpicture}
\end{center}

Using the Ptolemy relation on the crossing of $(u,v)$ and $(x,y)$, we have

\begin{align*}
f_\D(u,v) &= f_\D(u,v)f_\D(x,y) = f_\D(u,x)f_\D(v,y) + f_\D(u,y) f_\D(v,x)\\
&= \ell_{i-1} d_j + \ell_{i}\ell_j = \ell_{i-1}(d_j + \ell_{j+1}).
\end{align*}

Indeed, $(u,x), (u,y),$ and $(v,x)$ are tower arcs, and $f_{\D}(v,y) = d_j$ by definition.

Since the norm function is multiplicative and $N(\ell_{i-1}) = 1$, we have $N(f_\D(u,v)) = N(d_j + \ell_{j+1})$. Then, by separately evaluating cases where $j$ is even or odd and using Proposition \ref{prop:relateSnDnContFrac}, we can show in each case that $N(d_j + \ell_{j+1}) > 1$, which implies that $f_\D(u,v)$ is not a unit. 

The remaining cases can be shown with similar calculations.


\end{proof}

Even though Pell arcs can appear during the Triangulation Algorithm, the next result shows that they will not be present in a unitary triangulation of a polygon with respect to a frieze from a type 3 dissection. 

\begin{lemma}\label{lem:AlgFailsWithPell}
The Triangulation Algorithm cannot successfully produce a unitary triangulation if in at least one step it adds a Pell arc to the set $T$.
\end{lemma}

\begin{proof}

As in Lemma \ref{lem:BasicTriangleNotBoundary}, we can write a finite list of cases for how each vertex can appear in a triangle with one side a Pell arc and a second side which is either a Pell arc, a tower arc, or from $\D$. Then, we can compute that, in each case, the third arc of the triangle cannot have unit weight. We omit the full list for sake of brevity and instead provide a few example calculations. 

In some cases, we can compute this by direct calculation. Consider the triangulation in Figure \ref{fig:LemmaAlgFails1}; arc $(u,y)$ is a tower arc while arc $(y,v)$ is a Pell arc. We assume the tower containing $u$ is an $i$-tower, the tower containing vertex $x$ is a $j$-tower and the tower containing vertex $v$ is a $k$-tower for $i,j \geq 1$ and $k \geq 0$. 


\begin{figure}
\centering
\begin{tikzpicture}[scale = 0.6]
\draw (0,0) -- (0,4) -- (-2,6) -- (-2.4,7.5) -- (-0.9,7.5) -- (2,4.5) -- (1,3.5) -- (2,2.7) -- (2,1) -- (3,1) -- (5,-1) -- (5,-2) -- (4,-2) -- (2,0) -- (0,0);
\draw (0,4) -- (1,3.5) -- (0,2.7);
\draw (0,4) -- (1.25,5.25);
\draw (-2,6) -- (-0.9,7.5);
\draw (2,1) -- (2,0) -- (3,1);
\draw ( 5,-1) -- (4,-2);
\draw (0,2.7) -- (2,2.7);
\draw(0,1) -- (2,1);
\node[left] at (-2.4,7.5){$u$};
\node[right, xshift = 5pt] at (1,3.5){$y$};
\node[left] at (0,2.7){$x$};
\node[below] at (5,-2){$v$};
\draw [decorate,
    decoration = {brace}] (-0.7,7.7) --  (2.2,4.7);
\node[right, yshift = 5pt, xshift = 3pt] at (0.7,6.2){$i \geq 1$};    
\draw [decorate,
    decoration = {brace}] (-0.2,0) --  (-0.2,2.5); 
\node[left] at (-0.2,1.25){$j \geq 1$};    
\draw [decorate,
    decoration = {brace}] (3.2,1.2) --  (5.2,-0.8);    
\node[right, yshift = 5pt, xshift = 3pt] at (4.2,0.2){$k \geq 0$};    
\end{tikzpicture}
\caption{We can directly compute that $f_\D(u,v)$ is not a unit in such a region of a type 3 dissection}\label{fig:LemmaAlgFails1}
\end{figure}

If we apply the Ptolemy relation to the intersection of $(u,v)$ and $(x,y)$ and perform some simplifications, we find that \[
f_\D(u,v) = \ell_{i+k}(\ell_{j+1} + \ell_{j} + \ell_{j-1} + (d_j + \sqrt{2}d_{j-1})).
\]


It suffices to check whether the expression $\ell_{j+1} + \ell_{j} + \ell_{j-1} + (d_j + \sqrt{2}d_{j-1})$ is a unit. 
This expression clearly has the strict lower bound of $\ell_{j+1}$, so the smallest unit this could be is $\ell_{j+2}$. The expression would equal $\ell_{j+2}$ if and only if  \[
\sqrt{2}\ell_{j+1} = \ell_{j} + \ell_{j-1} + (d_j + \sqrt{2}d_{j-1}).
\]


We reduce both sides to be in terms of $s_{j-1}$ and $d_{j-1}$. The left hand side reduces to \[
\sqrt{2}\ell_{j+1} = (4 + 3\sqrt{2}) \ell_{j-1} = (4 + 3\sqrt{2})(s_{j-1} + d_{j-1}),\]
while the righthand side reduces to   \[
 \ell_{j+1} + \ell_j + d_j + \sqrt{2} d_{j-1} = (3+\sqrt{2}) s_{j-1} + (2+3\sqrt{2}) d_{j-1}
\]

We see here that $\sqrt{2} \ell_{j+1} > \ell_{j+1} + \ell_j + (d_j + \sqrt{2}d_{j-1})$. We have that $\ell_{j+2} > \ell_{j+1}+\ell_{j} + \ell_j + (d_j + \sqrt{2}d_{j-1}) > \ell_{j+1}$ and we know there are no units strictly between $\ell_{j+1}$ and $\ell_{j+2}$. Therefore, the original expression for $f_\D(u,v)$ cannot be equal to any unit.





\begin{figure}
\centering
\begin{tikzpicture}[scale = 0.6]
\draw (0,0) -- (-1,1) -- (-1,2.5) -- (-2,3) -- (-3,4.5) -- (-3,5.5) -- (-2,5) -- (-1,3.5) -- (1,3.5) -- (1,1) -- (2,0) -- (1,-1) -- (2.5,-2.5) -- (5,-2.5) -- (6,-3) -- (5,-3.5) -- (3.5,-3.5) -- (2.5,-4.5) -- (0,-2) -- (0,0); 
\draw (-3,4.5) -- (-2,5);
\draw (-2,3) -- (-1,3.5) -- (-1,2.5);;
\draw (-1,2.5) -- (1,2.5);
\draw(-1,1) -- (1,1);
\draw (1,1) - -(0,0) -- (1,-1) - -(0,-2);
\draw(1.5,-3.5) -- (2.5,-2.5) -- (3.5,-3.5) -- (3.5,-2.5);
\draw (5,-3.5) -- (5,-2.5);
\node[above] at (-3,5.5){$u$};
\node[above, xshift = 2pt] at (-1,3.5){$z$};
\node[left, xshift = -2pt, yshift = -1] at (0,0){$y$};
\node[right] at (6,-3){$x$};
\end{tikzpicture} 
\caption{The value $f_\D(u,x)$ is a unit but this arc could never appear as we run the Triangulation Algorithm}\label{fig:UnitButWontAppear}
\end{figure}

We note there are some ways to combine two Pell arcs which produce another arc with unit weight, as shown in Figure \ref{fig:UnitButWontAppear}. One can compute that $f_{\D}(u,x)$ is a unit for such a configuration. However, it is a consequence of Lemma \ref{lem:BasicTriangleNotBoundary} and the aforementioned case-work involving triangles with one Pell arc and one tower arc that $(u,x)$ could only be produced in the Triangulation Algorithm if we already had Pell arcs $(u,y)$ and $(y,x)$. Moreover, the Pell arc $(u,y)$ would only exist in the triangulation if we already had $(y,z)$. But then we see that we will never include $(u,x)$ because $(u,x)$ and $(y,z)$ cross. 

If the pair of triangles next to $y$ were on the other side of the tower, so that the picture below $y$ is reflected across a vertical axis, then the corresponding arc could be reached in the Triangulation Algorithm. But in this case, computation shows that the arc will not have a unit weight.

These cases show that, even though we know that a Pell arc can bound a side of one triangle, we cannot find a second triangle along a Pell arc whose other sides have unit weight. Therefore, if we choose a Pell arc in the Triangulation Algorithm, the algorithm cannot possibly terminate with a unitary triangulation of the polygon. 

\end{proof}

The final piece of our proof of Theorem \ref{thm:TypeThreeNonTowerFails} is showing that it is equivalent to say that a dissection is not a gluing of towers and a dissection $\mathcal{D}$ does not admit a triangulation into tower arcs and arcs from $\mathcal{D}$.

\begin{lemma}\label{lem:towerTriangulation}
If an $n$-gon $P$ with a dissection $\mathcal{D}$ can be triangulated using only tower arcs and the arcs from dissection, then $\mathcal{D}$ must be a gluing of towers.

\end{lemma}

\begin{proof}


If it is possible to triangulate a polygon $P$ with dissection $\D$ with only tower arcs and arcs from $\D$, then we can see $\D$ as a gluing of towers as follows. For every set of tower arcs which share a common endpoint in a roof point, we take the set of subpolygons triangulated by these arcs to be one tower. The arcs from $\D$ glue these towers together. 


\end{proof}

We are now ready to prove the main result in this section.
\begin{proof}[Proof of Theorem \ref{thm:TypeThreeNonTowerFails}]

Consider an $(n+3)$-gon $P_0$, for $n > 0$, with a dissection $\D$ such that $\D$ is a type 3 dissection and $\D$ is not a gluing of towers. By Lemma \ref{lem:BasicTriangleAlongBoundary}, either the triangulation algorithm finds an arc $\tau_1$ which is a tower arc or an arc from $\D$, or the algorithm terminates at the first step. So suppose that the algorithm finds such an arc $\tau_1$. 

By Lemma \ref{lem:towerTriangulation}, at some point the algorithm must either terminate or use an arc which is not a tower arc or from $\D$. Suppose that at step $k$, the algorithm finds such an arc $\tau_k$ with $f_\D(\tau_k) \in \mathbb{Z}[\sqrt{2}]^\times$ where $\tau_k$ is not a tower arc nor an arc in $\D$; moreover, let $k$ be the minimal number such that this is true. By Lemma \ref{lem:BasicTriangleNotBoundary}, $\tau_k$ must be a Pell arc since $\tau_1,\ldots,\tau_{k-1}$ were tower arcs or arcs from $\D$. By Lemma \ref{lem:AlgFailsWithPell}, then we know the algorithm cannot terminate after we have chosen $\tau_k$.

\end{proof}

Theorems \ref{thm:Towers} and \ref{thm:TypeThreeNonTowerFails} combine to show that a unitary triangulation is only possible in a type 3 dissection if it can be decomposed into a set of towers. Type one dissections are trivial to check and Remark \ref{rem:typetwo} explained the case for type 2 dissections. 

\begin{cor}\label{cor:TypeThreeConjHolds}
 Conjecture \ref{conj:UnitaryConj} holds for type 1,2, and 3 dissections.   
\end{cor}

\section*{Acknowledgements}
This project was largely completed at the University of Minnesota School of Mathematics Summer 2021 REU program and was partially supported by NSF RTG grant DMS-1148634 and NSF grant DMS-1949896.  The authors would like to thank Vic Reiner for organizing the REU and Trevor Karn for helping with computations on Sage. We also thank the anonymous referee for helpful comments regarding the exposition of the article. 

\printbibliography

\newpage
\appendix

\section{Basic triangles with tower arcs} \label{sec:AppA}



Here we show all options for how an arc in a type three dissection can form a triangle with either two tower arcs or a tower arc and a boundary arc. In each diagram, each tower is allowed to be any size such that the configuration is possible. For example, the left tower in Diagram 17 must be at least a 1-tower. 

\begin{center}
\begin{tabular}{c|c|c|c}

   \includegraphics[scale = 0.25]{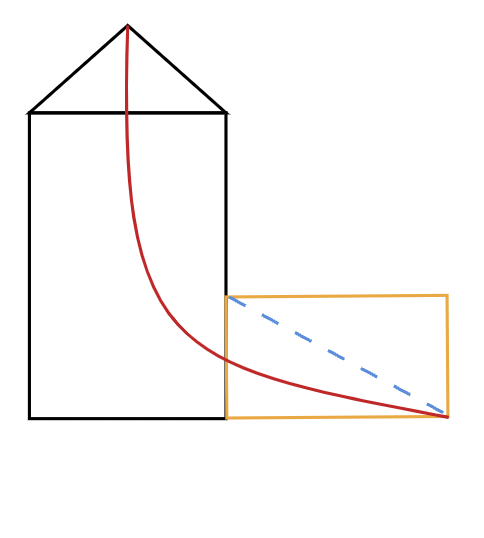}
   &
   \includegraphics[scale = 0.2]{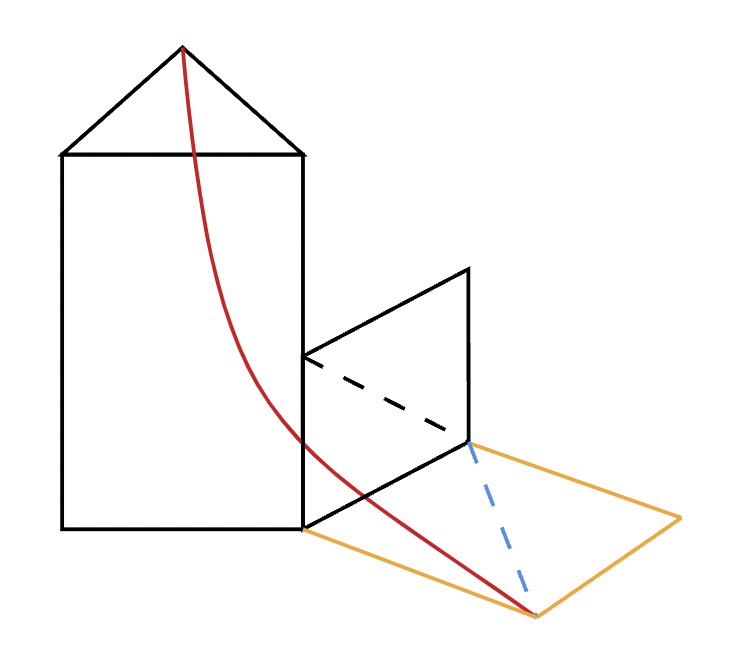}
     & \includegraphics[scale = 0.2]{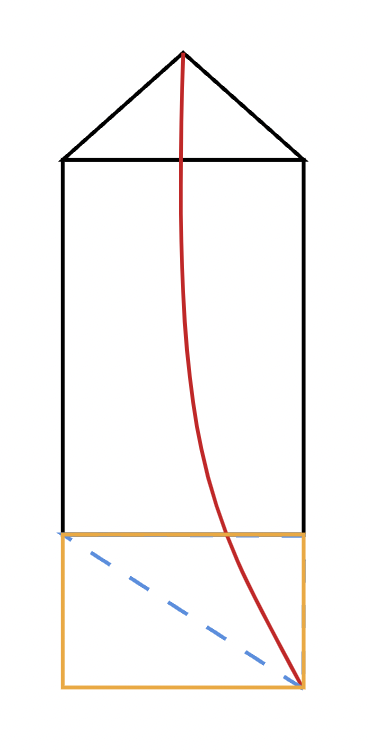} &\includegraphics[scale = 0.2]{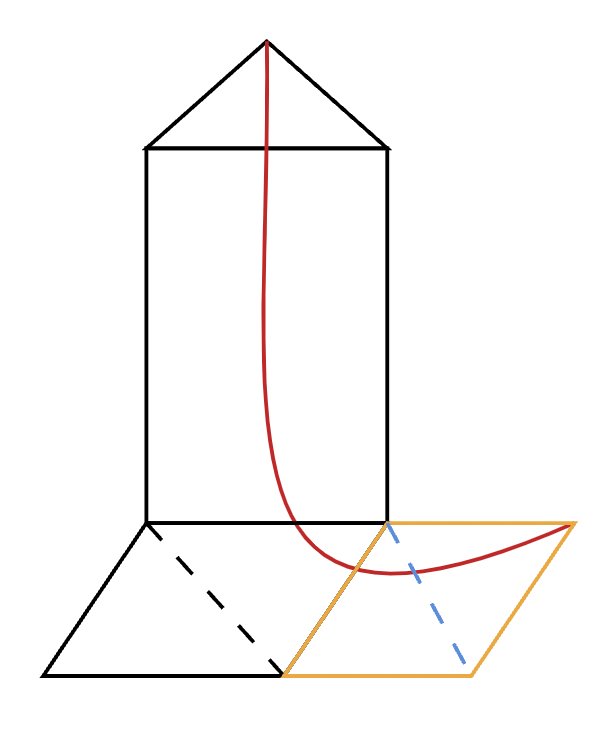}\\1&2&3&4 \\\hline
        \includegraphics[scale = 0.2]{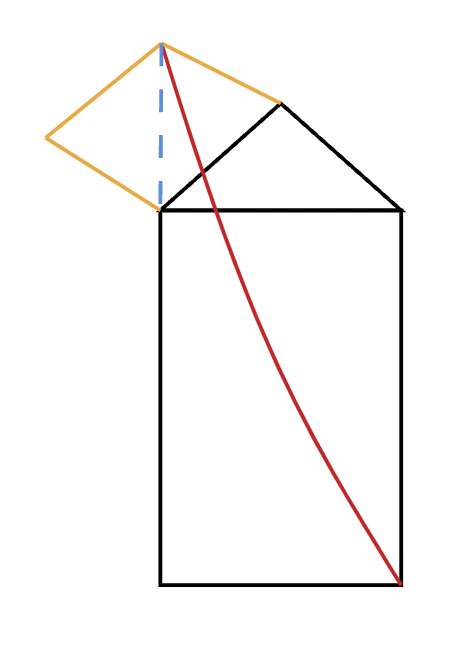} & \includegraphics[scale = 0.2]{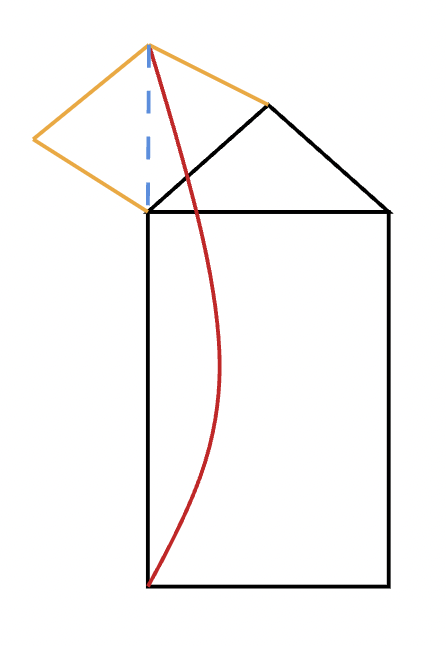} & \includegraphics[scale = 0.2]{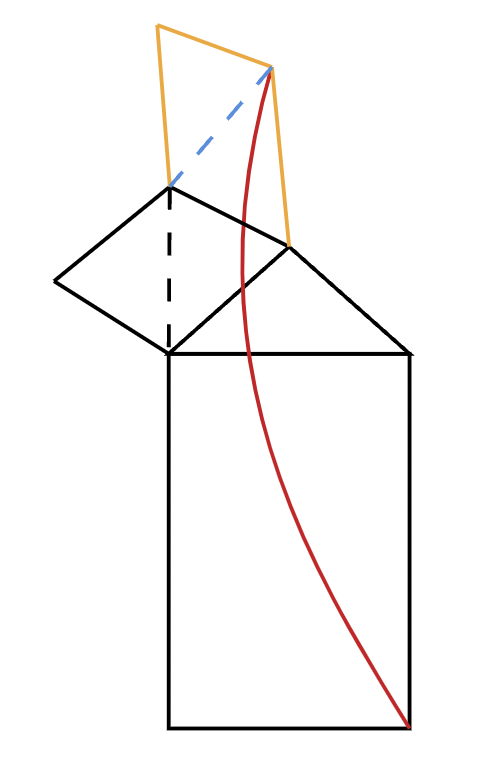}&\includegraphics[scale = 0.2]{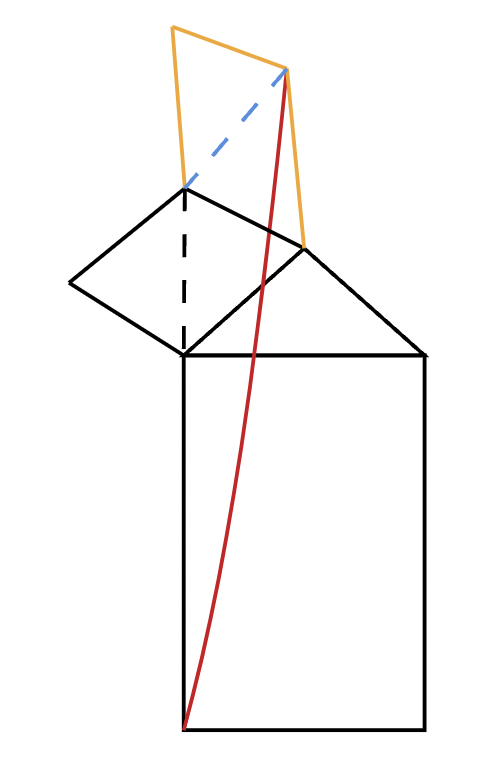}\\5&6&7&8\\\hline\includegraphics[scale = 0.17]{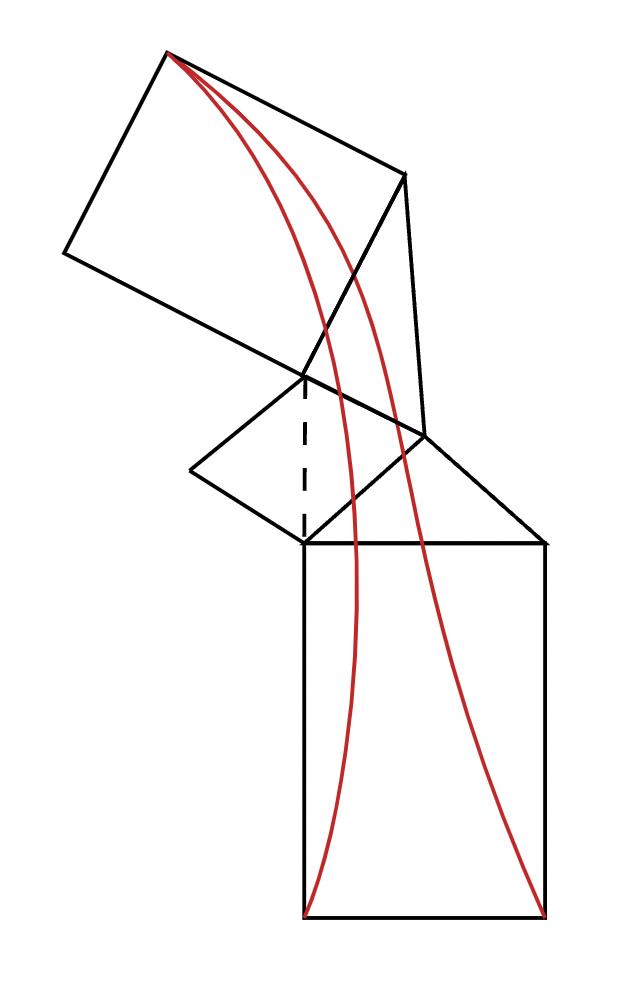}&\includegraphics[scale = 0.17]{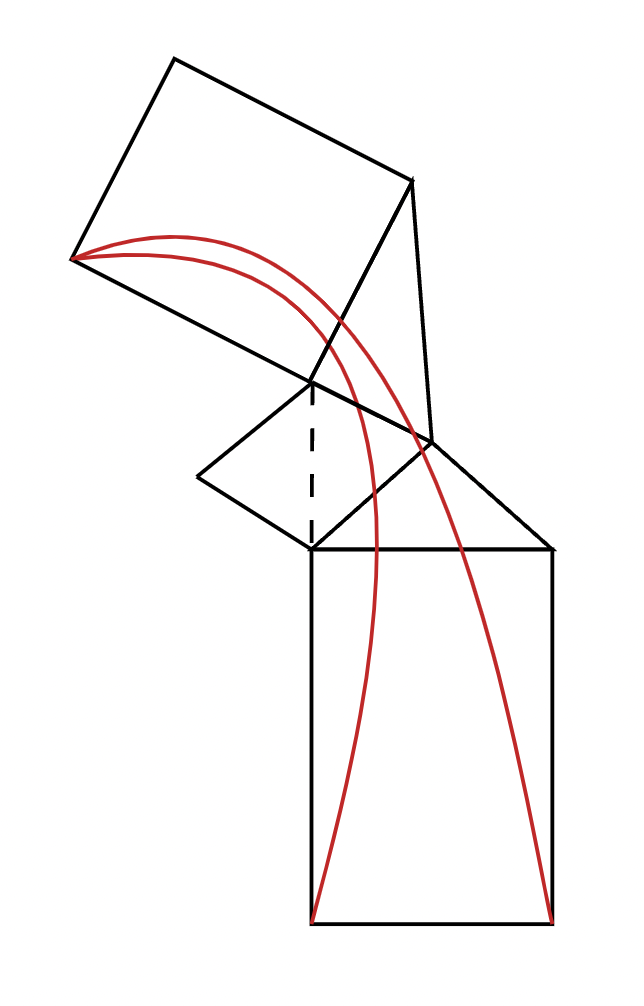}&\includegraphics[scale = 0.17]{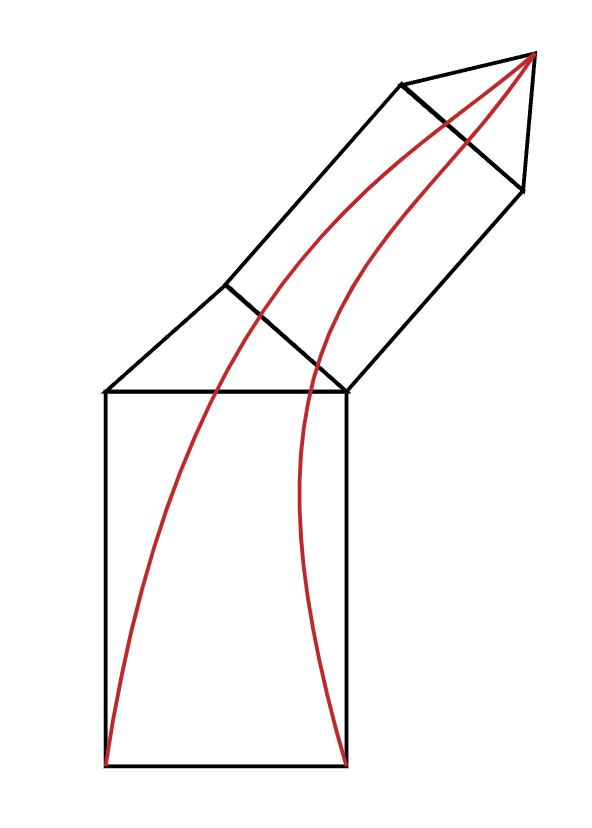}&\includegraphics[scale = 0.17]{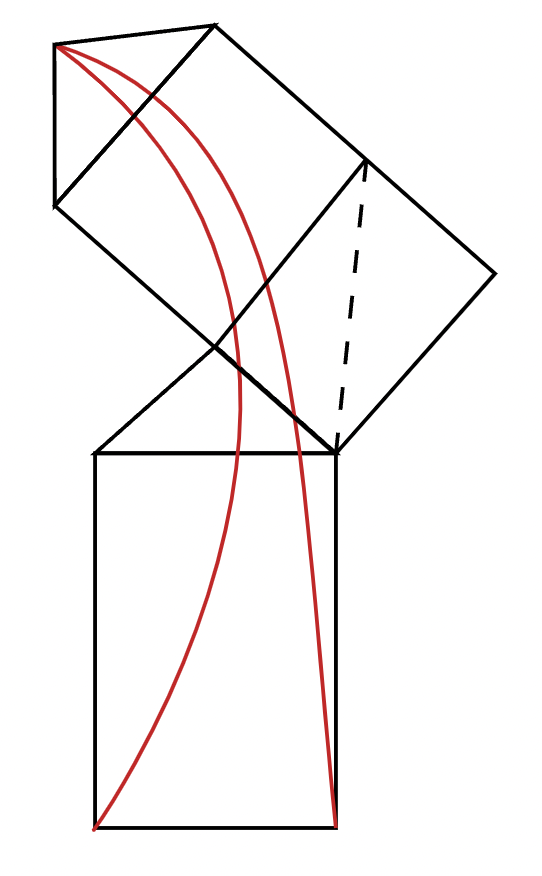}\\9&10&11&12\\\hline\includegraphics[scale = 0.17]{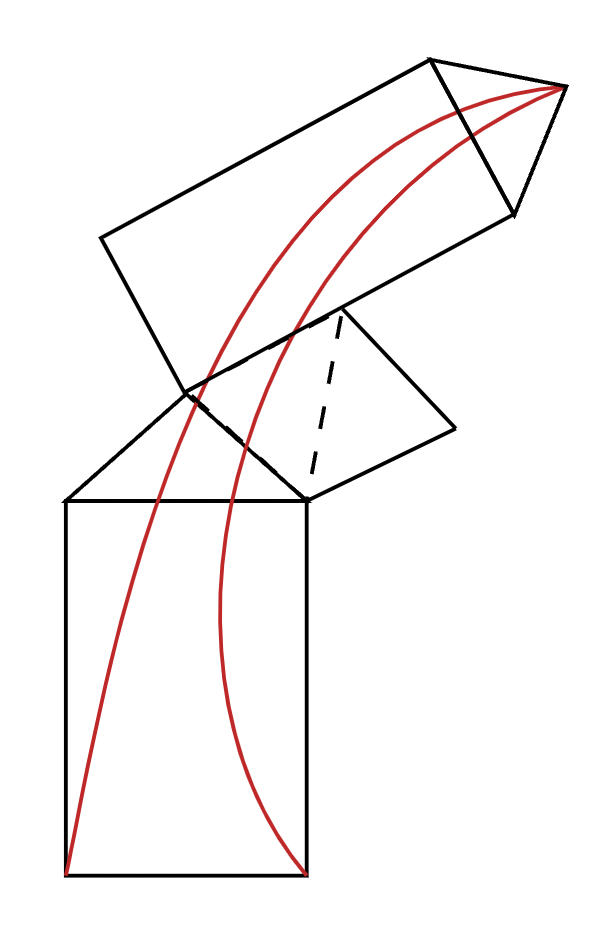}&\includegraphics[scale = 0.17]{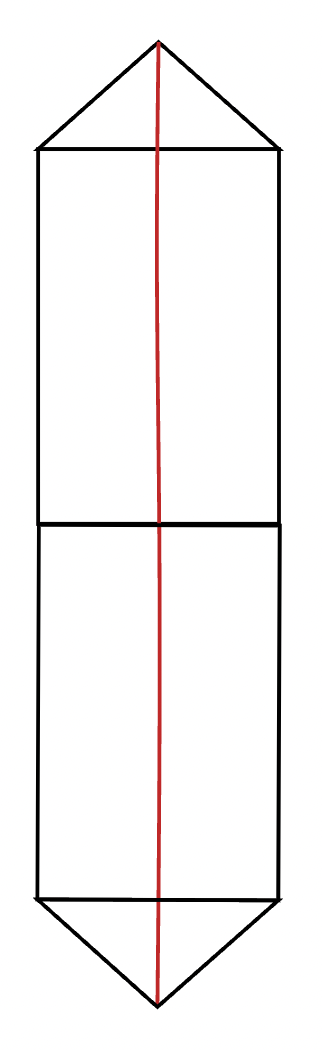}&\includegraphics[scale = 0.17]{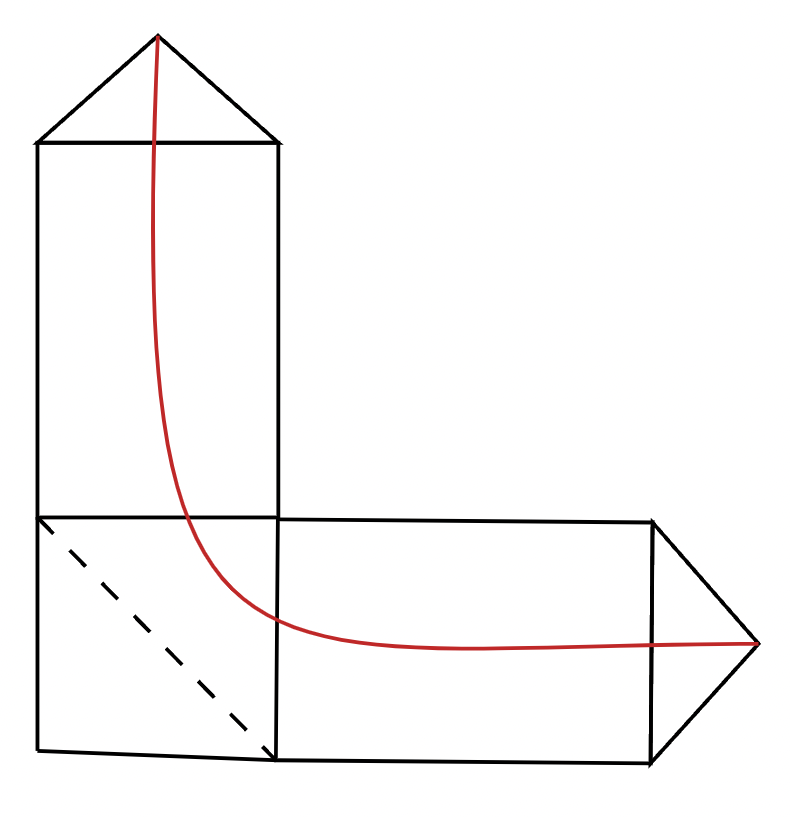}&\includegraphics[scale = 0.17]{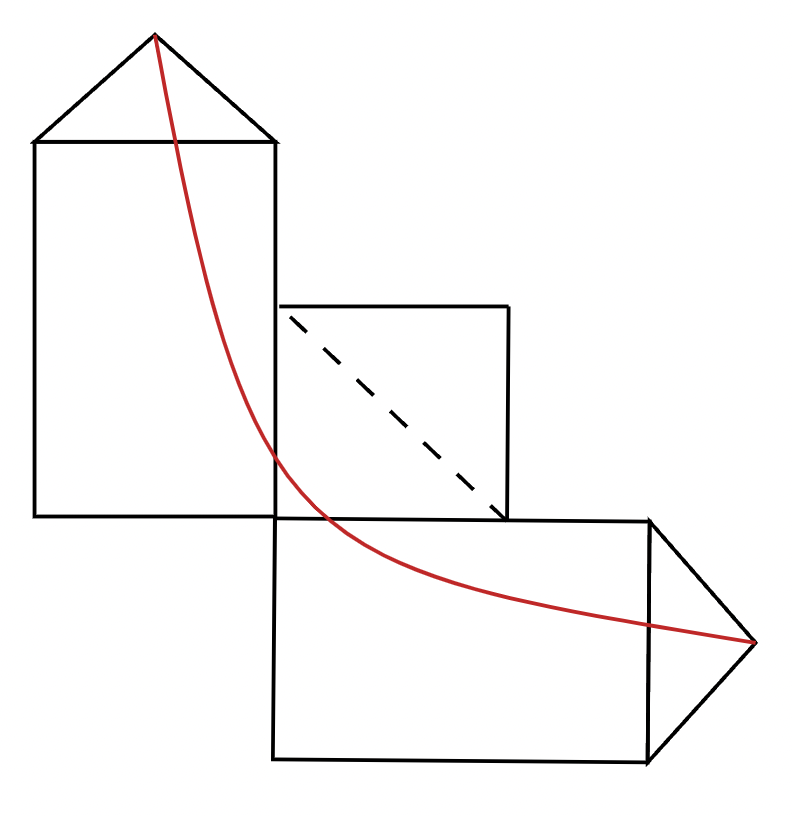}\\13&14&15&16\\\hline\includegraphics[scale = 0.17]{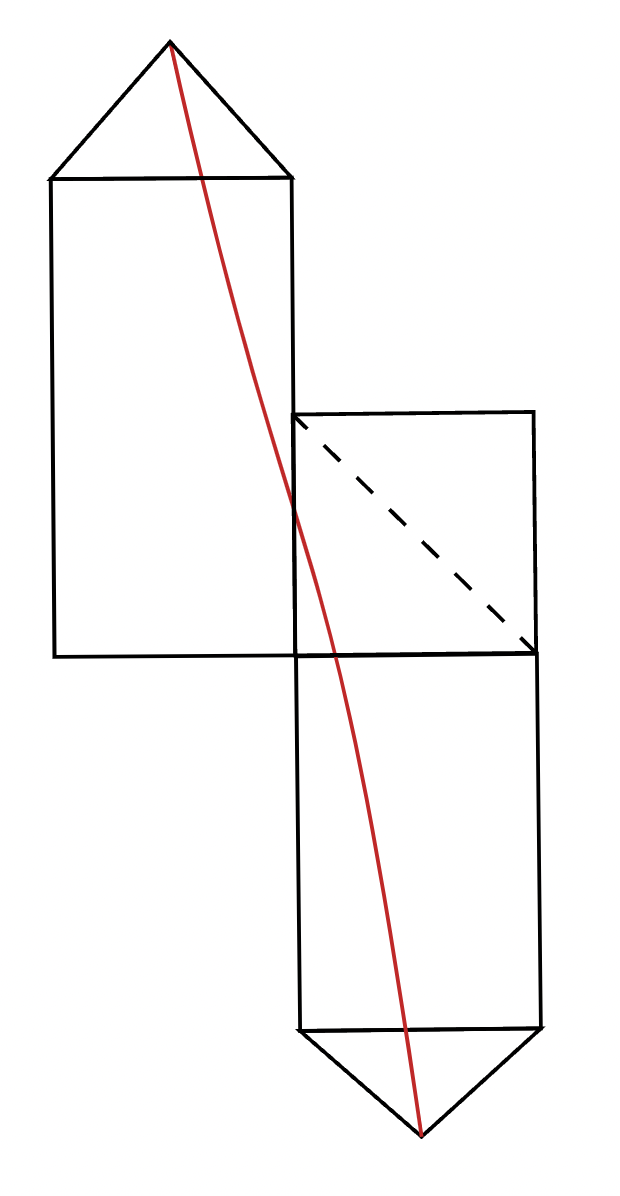}\\17

\end{tabular}
\end{center}
\newpage
Note that each yellow quadrilateral denotes a quadrilateral whose boundary is used to form a triangle with the new arc. The dotted lines denote replacing the quadrilateral with a triangle to produce a distinct case.  
\begin{center}
\includegraphics[scale = 0.27]{Diagrams/12.png}
\end{center}

For example, in Picture 2 (also shown below), we have one quadrilateral colored in yellow, and inside the yellow quadrilateral, there is blue dotted line. The blue dotted line denotes that the quadrilateral can be replaced by a triangle, which gives us another new arc as shown below. 

\begin{center}
\includegraphics[scale = 0.35]{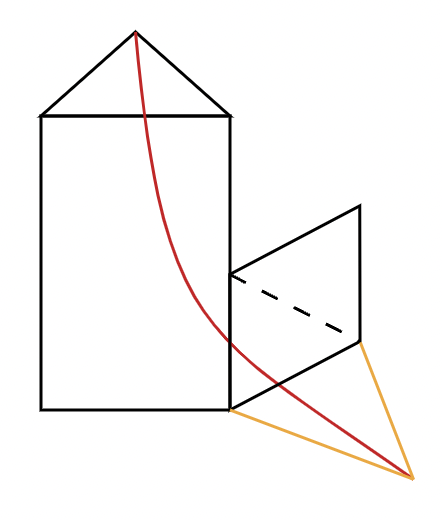}
\end{center}

Picture 2 also has a black dotted line in the black quadrilateral, which means that a triangle can replace the black quadrilateral, and the resulting new arc is a new case as shown below. 
\begin{center}
\includegraphics[scale = 0.35]{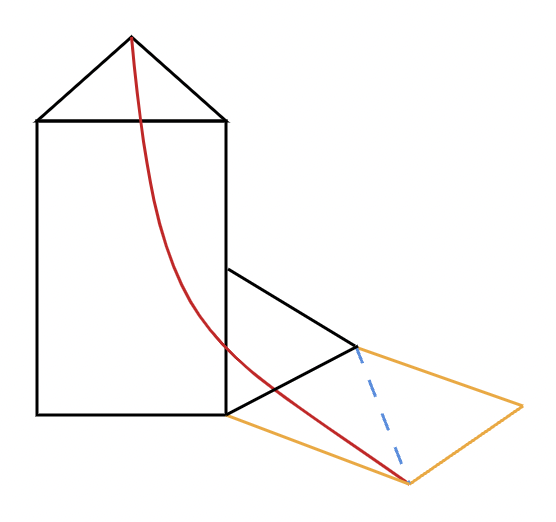}
\end{center}

Therefore, Picture 2 encapsulates four distinct cases of a new arc produced by a tower arc and a subpolygon in total, even though Picture 2 only contains one figure. This explains why the 17 figures capture all 40 possibilities. 

As was described in Lemma \ref{lem:BasicTriangleNotBoundary}, the only arcs from this table which would have unit weight are entry 2 where both optional shapes are triangles, entry 3 when the optional shape is a quadrilateral and entry 17 when the optional shape is a triangle. 

\end{document}